\documentclass[12pt]{amsart}

\input{xy}
\xyoption{all}
\usepackage{graphicx}
\usepackage{color}
\usepackage{verbatim}
\usepackage{hyperref}

\newtheorem{theorem}{Theorem} [section]
\newtheorem{prop}[theorem]{Proposition}
\newtheorem{lemma}[theorem]{Lemma}
\newtheorem{cor}[theorem]{Corollary}

\newtheorem{question}[theorem]{Question}

\theoremstyle{definition}

\newtheorem{example}{Example}

\theoremstyle{remark}
\newtheorem*{remark}{Remark}

\numberwithin{equation}{section}
\numberwithin{figure}{section}
\numberwithin{example}{section}

\newcommand\A{{\mathbb A}}
\newcommand\C{{\mathbb C}}
\newcommand\CC{{\mathbb C}}

\newcommand\N{{\mathbb N}}

\renewcommand\P{{\mathbb P}}
\newcommand\PP{{\mathbb P}}
\newcommand\R{{\mathbb R}}
\newcommand\RR{{\mathbb R}}

\newcommand\Q{{\mathbb Q}}

\newcommand\eps{\varepsilon}

\renewcommand\phi{\varphi}

\newcommand\iso{\simeq} 

\newcommand\cM{\mathcal{M}}

\newcommand\Gal{\operatorname{Gal}}

\newcommand\supp{\operatorname{supp}}   

\renewcommand\Re {\operatorname{Re}}

\renewcommand\Im {\operatorname{Im}}

\newcommand\capacity {\operatorname{cap}} 
\newcommand\Res {\operatorname{Res}} 

\newcommand\M {\mathrm{M}}

\newcommand\Per {\mathrm{Per}}
\newcommand\Preper {\mathrm{Preper}}
\newcommand\Bif {\mathrm{Bif}}

\newcommand\kbar {\overline{k}}
\newcommand\Qbar {\overline{\Q}}
\newcommand\kvbar {\overline{k}_v}

\newcommand\Berk {\operatorname{Berk}}

\newcommand\hhat {\hat{h}}

\begin{document}

\title{Bifurcation measures and quadratic rational maps}

\author{Laura De Marco}
\address{Department of Mathematics, Northwestern University, USA}
\email{demarco@math.northwestern.edu}
\author{Xiaoguang Wang}
\address{Department of Mathematics, Zhejiang University, P.R.China} 
\email{wxg688@163.com}
\author{Hexi Ye}
\address{Department of Mathematics, University of British Columbia, Canada}
\email{yehexi@math.ubc.ca}


\date{\today}

\begin{abstract}
We study critical orbits and bifurcations within the moduli space $\M_2$ of quadratic rational maps, $f: \P^1\to \P^1$.  We focus on the family of curves, $\Per_1(\lambda) \subset \M_2$ for $\lambda\in\C$, defined by the condition that each $f\in \Per_1(\lambda)$ has a fixed point of multiplier $\lambda$.  We prove that the curve $\Per_1(\lambda)$ contains infinitely many postcritically-finite maps if and only if $\lambda=0$, addressing a special case of \cite[Conjecture 1.4]{BD:polyPCF}.  We also show that the two critical points of $f$ define distinct bifurcation measures along $\Per_1(\lambda)$.
\end{abstract}


\thanks{The research was supported by the National Science Foundation.}

\maketitle

\thispagestyle{empty}

\section{Introduction}

In this article, we study the dynamics of holomorphic maps $f: \P^1_\C \to \P^1_\C$ of degree $2$.  We concentrate our analysis on the lines $\Per_1(\lambda)$ within the moduli space $\M_2\iso\C^2$ of quadratic rational maps, introduced by Milnor in \cite{Milnor:quad}.  For each $\lambda\in\C$, $\Per_1(\lambda)$ is the set of all (conformal conjugacy classes of) maps $f$ with a fixed point $p$ at which $f'(p) = \lambda$; so $\Per_1(0)$ is the family of maps conjugate to a polynomial.

Our first main result addresses a special case of Conjecture 1.4 of \cite{BD:polyPCF}.  (See also the corrected version in \cite[\S6.1]{D:stableheight} and this case presented in \cite[\S6.5]{Silverman:moduli}.)  The conjecture aims to classify the algebraic subvarieties of the moduli space $\M_d$ containing a Zariski-dense set of postcritically-finite maps, for each degree $d\geq 2$.  By definition, a rational map of degree $d$ is postcritically finite if each of its $2d-2$ critical points has a finite forward orbit.  It is known that the postcritically finite maps form a Zariski-dense subset of $\M_d$ in every degree $d\geq 2$, but the subvarieties intersecting many of them are expected to be quite special.  

\begin{theorem}  \label{PCF maps}
The curve $\Per_1(\lambda)$ in $\M_2$ contains infinitely many postcritically-finite maps if and only if $\lambda = 0$.
\end{theorem}

\noindent
This result is the exact analog of Theorem 1.1 in \cite{BD:polyPCF} that treated cubic polynomials.  As in that setting, one implication is easy:  if $\lambda = 0$, the curve $\Per_1(0)$ defines the family of quadratic polynomials, and it contains infinitely many postcritically-finite maps (by a standard application of Montel's theorem on normal families).  The converse direction is more delicate; its proof, though similar in spirit to that of \cite[Theorem 1.1]{BD:polyPCF}, required different techniques, more in line with our work for the Latt\`es family of \cite{DWY:Lattes}.

The second theme of this paper is a study of the bifurcation locus in the curves $\Per_1(\lambda)$; refer to \S\ref{bifurcation} for definitions.  For each fixed $\lambda\not=0$, we work with an explicit parametrization of $\Per_1(\lambda)^{cm}$, a double cover of $\Per_1(\lambda)$ consisting of maps with marked critical points:
	$$f_t(z) = \frac{\lambda z}{z^2 + t z + 1}$$
with $t\in\C$.  The map $f_t$ has a fixed point at $z=0$ with multiplier $\lambda$; the critical points of $f_t$ are $\{\pm 1\}$ for all $t$; note that $f_t$ is conjugate to $f_{-t}$ via the conjugacy $z\mapsto -z$ interchanging the two critical points.  Each critical point determines a finite bifurcation measure on $\Per_1(\lambda)^{cm}$, which we denote by $\mu^+_\lambda$ and $\mu^-_\lambda$.  (The symmetry of $f_t$ implies that $\mu^-_\lambda = A_* \mu^+_\lambda$ for $A(t) = -t$.)  Our main result in this direction is:

\begin{theorem}  \label{distinct measures}
For every $\lambda\not=0$, we have $\mu^+_\lambda \not = \mu^-_\lambda$ in $\Per_1(\lambda)^{cm}$.
\end{theorem}

\noindent
Theorem \ref{distinct measures} is not unexpected.  For any $\lambda$, the two critical points should behave independently.  In fact, it is not difficult to show that the critical points {\em cannot} satisfy any dynamical relation of the form $f^n(+1) \equiv f^m(-1)$ along $\Per_1(\lambda)^{cm}$; see Corollary \ref{independence}.  However, computational experiments suggested some unexpected alignment of the two bifurcation loci, $\mathrm{Bif}^+ = \supp\mu^+_\lambda$ and $\mathrm{Bif}^- = \supp \mu^-_\lambda$, for certain values of $\lambda$.  For example, for values of $\lambda$ near $-4$,   the two bifurcation loci appear remarkably similar.   See Figure \ref{L=-4} and Question \ref{distinct supports}.

A key ingredient in our proof of Theorem \ref{PCF maps} is an equidistribution statement, that parameters $t$ where the critical point $\pm 1$ has finite forward orbit for $f_t$ will be uniformly distributed with respect to the bifurcation measure $\mu^{\pm}_\lambda$.   Post-critically finite maps have algebraic multipliers, so it suffices to study the case where $\lambda\in\Qbar$; to prove the equidistribution result, we rely on the arithmetic methods introduced in \cite{Baker:Rumely:equidistribution} and \cite{FRL:equidistribution}.  But there are two features of $\Per_1(\lambda)$ that distinguish it from a series of recent articles on this theme (see e.g. \cite{BD:preperiodic, BD:polyPCF, Ghioca:Hsia:Tucker, GHT:preprint, Favre:Gauthier}); in particular, we could {\em not} directly apply the existing arithmetic equidistribution theorems for points of small height on $\P^1$.  
\begin{enumerate}
\item		The bifurcation locus can be noncompact, and the proof that the potential functions for the bifurcation measures are continuous across $t=\infty$ is more delicate (we show this in Theorem \ref{convergence}, with the method we used in \cite{DWY:Lattes}); and
\item		the canonical height function defined on $\Per_1(\lambda)$ (associated to each critical point) is only ``quasi-adelic,"  meaning that it may have nontrivial contributions from infinitely many places of any number field containing $\lambda$.  
\end{enumerate}
Because of (2), we use a modification of the original equidistribution result (and of its proofs, following \cite{BRbook, FRL:equidistribution}) that appears in \cite{Ye:quasi}.   We deduce the following result.  (The full statement of this theorem appears as Theorem \ref{equidistribution at all places}.)

\begin{theorem}  \label{equidistribution}
For every $\lambda \in \Qbar\setminus\{0\}$ with $\lambda$ not a root of unity, or for $\lambda=1$, the set
	$$\Preper^+_\lambda = \{t\in\Per_1(\lambda)^{cm}:  + 1 \mbox{ has finite forward orbit for } f_t\}$$
 is equidistributed with respect to $\mu^+_\lambda$; similarly for $\Preper^-_\lambda$ and $\mu^-_\lambda$.  More precisely, for any non-repeating sequence of finite sets $S_n \subset \Preper^+_\lambda$, the discrete probability measures
	$$\mu_n \; = \; \frac{1}{|G \cdot S_n|} \; \sum_{t \,\in \,G\cdot S_n} \; \delta_t$$
converge weakly to the measure $\mu^+_\lambda$, where $G= \Gal(\overline{\Q(\lambda)}/\Q(\lambda))$.
\end{theorem}

\noindent
Note that the sets $\Preper^{\pm}_\lambda$ are invariant under the action of the Galois group $G$:  if $+1$ is preperiodic for $t_0$, then $+1$ is preperiodic for all $t$ in its Galois orbit, since these parameters are solutions of an equation of the form $f_t^n(+1) = f_t^m(+1)$, with coefficients in $\Q(\lambda)$.  A ``classical" setting of Theorem \ref{equidistribution} would be to take $S_n$ as the full set of solutions to the equation $f_t^n(+1) = f_t^m(+1)$, with any sequence $0 \leq m = m(n) < n$ as $n\to\infty$.

The equidistribution of Theorem \ref{equidistribution} for $\lambda = 0$ is well known.  It was first shown by Levin (in the classical sense of equidistribution) \cite{Levin:iteration}, and it was shown in the stronger (arithmetic) form by Baker and Hsia \cite[Theorem 8.15]{Baker:Hsia}.   In fact the equidistribution of Theorem \ref{equidistribution} holds at each place $v$ of the number field $\Q(\lambda)$, on an appropriately-defined Berkovich space $\P^{1,an}_{\C_v}$, for sets $S_n$ with canonical height tending to 0; see Theorem \ref{equidistribution at all places}.

\medskip\noindent{\bf Outline of the article.}
In Section \ref{bifurcation}, we introduce the families $\Per_1(\lambda)^{cm}$, the bifurcation loci within these curves, and the bifurcation measures $\mu^+_\lambda$ and $\mu^-_\lambda$.  We prove the independence of the critical points (Corollary \ref{independence}) and pose Question \ref{distinct supports} about the bifurcation loci.  In Section \ref{measures}, we give the proof of Theorem \ref{distinct measures}.  In Section \ref{homogeneous}, we prove that the measures $\mu^+_\lambda$ and $\mu^-_\lambda$ have continuous potentials on all of $\P^1$, assuming that $\lambda$ is not ``too close" to a root of unity (Theorem \ref{convergence}).  In Section \ref{non-archimedean}, we prove a non-archimedean convergence statement, analogous to Theorem \ref{convergence}, for $\lambda\in\Qbar$ that are not equal to roots of unity.  In Section \ref{sets}, we introduce the homogeneous bifurcation sets and compute their homogeneous capacities.  In Section \ref{equidistribution section}, we prove the needed equidistribution theorems, including Theorem \ref{equidistribution}.  In Section \ref{proof section}, we complete the proof of Theorem \ref{PCF maps}.

\medskip\noindent{\bf Acknowledgements.}  We would like to thank Ilia Binder, Dragos Ghioca, and Curt McMullen for helpful comments.  We also thank Suzanne Boyd for help with her program Dynamics Explorer, used to generate all images in this article.

\bigskip
\section{Bifurcation locus in the curve $\Per_1(\lambda)$}
\label{bifurcation}

The moduli space $\M_2$ is the space of conformal conjugacy classes of quadratic rational maps $f: \P^1_\C \to \P^1_\C$, where two maps are equivalent if they are conjugate by a M\"obius transformation; see \cite{Milnor:quad, Silverman}.
In this section, we provide some basic results about the bifurcation locus within the curves $\Per_1(\lambda)$ in $\M_2$.   By definition, $\Per_1(\lambda)$ is the set of conjugacy classes of quadratic rational maps with a fixed point of multiplier $\lambda$.  In Milnor's parameterization of $\M_2\iso \C^2$, using the symmetric functions in the three fixed point multipliers, each $\Per_1(\lambda)$ is a line \cite[Lemma 3.4]{Milnor:quad}.  For $\lambda =0$, $\Per_1(0)$ is the family of quadratic polynomials, usually parametrized by $f_t(z) = z^2 +t$ with $t\in\C$.  

\begin{remark}  A number of results have appeared since \cite{Milnor:quad} that address features of the bifurcations within $\Per_1(\lambda)$.  For example, when $|\lambda| < 1$, it is known that the bifurcation locus is homeomorphic to the boundary of the Mandelbrot set; this follows from the straightening theorem of \cite{Douady:Hubbard} (see the remark following Corollary 3.4 of \cite{Goldberg:Keen:shift}).  See \cite{Petersen:elliptic, Uhre:model, Buff:Epstein:Ecalle} for more in the setting of $\lambda$ a root of unity.  Berteloot and Gauthier have recently studied properties of the bifurcation current on $\M_2$ near infinity \cite{Berteloot:Gauthier}.
\end{remark}

\subsection{Bifurcations.}  \label{bifurcation definition}
Let $X$ be a complex manifold.  A {\em holomorphic family} of rational maps parametrized by $X$ is a holomorphic map 
	$$f: X\times\P^1\to \P^1.$$
We often write $f_t$ for the restriction $f(t, \cdot): \P^1\to \P^1$ for each $t\in X$.  A holomorphic family $\{f_t, t\in X\}$ of rational functions of degree $d\geq 2$ is {\em stable at $t_0\in X$} if the Julia sets $J(f_t)$ are moving holomorphically in a neighborhood of $t_0$.  In particular, $f_{t_0}|J(f_{t_0})$ is topologically conjugate to all nearby maps when restricted to their Julia sets (and the Julia sets are homeomorphic) \cite{Mane:Sad:Sullivan, McMullen:CDR}.  An equivalent characterization of stability, upon passing to a branched cover of $X$ where the critical points $c_1, c_2, \ldots, c_{2d-2}$ can be labelled holomorphically, is that the sequence of holomorphic maps
	$$\{t \mapsto f_t^n(c_i(t))\}$$
forms a normal family for each $i$ on some neighborhood of $t_0$.  The failure of normality can be quantified with the construction of a positive $(1,1)$-current on the parameter space $X$, as follows.

Suppose that we can express $f_t$ in homogeneous coordinates, as a holomorphic family
	$$F_t: \C^2 \to \C^2$$
for $t\in X$.  Assume we are given holomorphic functions $\tilde{c}_i: X\to \C^2\setminus\{(0,0)\}$ projecting to the critical points $c_i(t)\in\P^1$ of $f_t$.  We define the {\em bifurcation current} of $c_i$ on $X$ by
	$$T_i := dd^c \left(\lim_{n\to\infty} \frac{1}{d^n} \log \| F_t^n(\tilde{c}_i(t)) \| \right) .$$
The current vanishes if and only if the family $\{t \mapsto f_t^n(c_i(t))\}$ is normal.  In particular, $T_i=0$ for all $i$ if and only if the family is stable.  In fact, the family $F_t$ and the functions $\tilde{c}_i$ can always be defined locally on $X$, after passing to a branched cover where the critical points can be labelled, and the current $T_i$ is independent of the choice of $\tilde{c}_i$ and the homogenization $F_t$.  When the parameter space $X$ has dimension 1, note that the current $T_i$ is a measure (where $dd^c$ is simply the Laplacian), and we will refer to it as the {\em bifurcation measure}.  (See  \cite{D:current, D:lyap, Dujardin:Favre:critical}.)

The {\em bifurcation locus} is the set of parameters in $X$ where $f_t$ fails to be stable.  It coincides with the union of the supports of the bifurcation currents $T_1, \ldots, T_{2d-2}$.  The following lemma is a straightforward application of Montel's theorem; for a proof of Montel's theorem, see \cite[\S3]{Milnor:dynamics}.

\begin{lemma} \label{activity}
Let $f: X\times\P^1\to \P^1$ be a holomorphic family of rational functions of degree $>1$, with marked critical point $c: X\to\P^1$.  Let $T$ be the bifurcation current of $c$, and assume that $T\not= 0$.  Then there are infinitely many parameters $t\in X$ where $c(t)$ has finite orbit for $f_t$.
\end{lemma}

\proof
Fix $t_0\in \supp T$.  Choose any repelling periodic cycle for $f_{t_0}$ of period $\geq 3$ that is not in the forward orbit of $c(t_0)$.  By the Implicit Function Theorem, the repelling cycle persists in a neighborhood $U$ of $t_0$.  If the orbit of the critical point $c(t)$ were disjoint from the cycle for all $t\in U$, then Montel's Theorem would imply that $\{t\mapsto f^n_t(c(t))\}$ forms a normal family on $U$.  This contradicts the fact that $t_0$ lies in the support of $T$. Consequently, there is a parameter $t_1\in U\setminus\{t_0\}$ where $c(t_1)$ is preperiodic for $f_{t_1}$.  Shrinking the neighborhood $U$, we obtain an infinite sequence of such parameters converging to $t_0$.
\qed

\begin{example}  \label{polynomials}
For the family of quadratic polynomials, $f_t(z) = z^2 + t$, there is only one critical point (at $z=0$) inducing bifurcations.  The associated bifurcation current defines a measure on the parameter space ($t\in\C$).  It is equal, up to a normalization factor, to the harmonic measure supported on the boundary of the Mandelbrot set $\mathcal{M}$ \cite[Example 6.1]{D:current}.  In this case, there is no need to pass to homogeneous coordinates; a potential function for the normalized bifurcation measure is given by
\begin{equation}  \label{quadratic G}
	G_{\mathcal{M}}(t) = \lim_{n\to\infty} \frac{1}{2^n} \log^+ |f_t^n(t)|
\end{equation}
for $t\in\C$.
\end{example}

\subsection{The bifurcation locus in the critically-marked curve.}
\label{definitions}
Fix a complex number $\lambda \not=0$, and set
	$$f_{\lambda,t}(z) = \frac{\lambda z}{z^2 + tz + 1}$$
for all $t\in\C$.  (We will often write $f_t$ for $f_{\lambda,t}$ when the dependence on $\lambda$ is clear.)  Then $f_t$ has critical points at $c_+(t) = +1$ and $c_-(t) = -1$ for all $t\in\C$.  Since $f_t$ is conjugate to $f_{-t}$ by $z\mapsto -z$, the family $f_t$ parametrizes a degree-2 branched cover of the curve $\Per_1(\lambda) \subset \M_2$, which we denote by $\Per_1(\lambda)^{cm}$; the $cm$ in the superscript stands for ``critically marked."  

We define the bifurcation currents associated to the critical points $c_+$ and $c_-$ as in \S\ref{bifurcation definition}.  As the parameter space is 1-dimensional, the currents are in fact measures; we denote these bifurcation measures by $\mu^+_\lambda$ and $\mu^-_\lambda$.  The supports will be denoted by
	$$\Bif^+ = \supp \mu^+_\lambda \qquad \mbox{and} \qquad \Bif^- = \supp \mu^-_\lambda.$$
These bifurcation measures have globally-defined potential functions.  We set
\begin{equation} \label{F_t}
	F_t(z_1, z_2) =  (\lambda z_1 z_2, z_1^2 + t z_1 z_2 +  z_2^2)
\end{equation}
and
\begin{equation} \label{H}
	H_\lambda^{\pm}(t) = \lim_{n\to\infty} \frac{1}{2^n} \log \| F_t^n(\pm 1, 1) \|,
\end{equation}
so that
	$$\mu^+_\lambda = \frac{1}{2\pi} \Delta H^+_\lambda \qquad \mbox{and} \qquad \mu^-_\lambda = \frac{1}{2\pi} \Delta H^-_\lambda.$$
Since $f_{-t}(z) = -f_t(-z)$, we see that
	$$H_\lambda^-(t) = H_\lambda^+(-t)$$
and $\Bif^- = -\Bif^+$.

The following proposition follows from the observations of Milnor in \cite{Milnor:quad}.

\begin{prop} \label{compactness}
The bifurcation loci $\Bif^+$ and $\Bif^-$ are nonempty for all $\lambda \not=0$.  They are compact in $\Per_1(\lambda)^{cm}$ if and only if $|\lambda| \not=1$ or $\lambda = 1$.

Moreover, when the bifurcation locus $\Bif = \Bif^+\cup \Bif^-$ is compact, the unbounded stable component consists of maps for which both critical points lie in the basin of an attracting (or parabolic, in the case of $\lambda=1$)  fixed point.
\end{prop}

\proof
The three fixed points of $f_t$ lie at $0$ and
	$$Z_{\pm}(t)  = \frac{-t \pm \sqrt{t^2-4(1-\lambda)}}{2}.$$
The set of fixed point multipliers is
	$$\{ \lambda, (1- Z_\pm(t)^2)/\lambda \}.$$
For each fixed $\lambda$, there are well-defined branches of the square root for $|t| >>0$ so that $Z_\pm$ define analytic functions near infinity, with $Z_+(t) \to 0$ and $Z_-(t) \to \infty$ as $t\to \infty$.  The set of fixed point multipliers converges to $\{\lambda, 1/\lambda, \infty\}$ as $t\to \infty$.  (Compare \cite[Lemma 4.1]{Milnor:quad}.)

We first observe that $\Bif^+$ and $\Bif^-$ are nonempty.  Note that the two fixed points $Z_+(t)$ and $Z_-(t)$ must collide at $t = \pm 2 \sqrt{1-\lambda}$.  Consequently, their multipliers are 1 at that point, while they cannot be persistently equal to 1.  By the characterizations of stability \cite[Theorem 4.2]{McMullen:CDR}, the parameters $t = \pm 2 \sqrt{1-\lambda}$ will lie in the bifurcation locus.

Suppose that $|\lambda|=1$ with $\lambda \not=1$.  Then 
	$$f_t'(Z_+(t)) =  \frac{1}{\lambda} \left( 1 - \frac{(1-\lambda)^2}{t^2} + O\left(\frac{1}{t^4}\right) \right)$$
for $t$ large.  Fixing any large value of $R>0$ and letting the argument of $t=Re^{i\theta}$ vary in $[0,2\pi]$, the absolute value of $(1 - (1-\lambda)^2/t^2)$ will fluctuate around 1.  Consequently, the multiplier $f_t'(Z_+(t))$ will have absolute value 1 for some parameter $t$ with $|t| =R$, for all sufficiently large $R$.   Again by the characterizations of stability, all such parameters will lie in the bifurcation locus.  Consequently, $\Bif^+$ and $\Bif^-$ are unbounded.

For $|\lambda| \not=1$, $f_t$ has an attracting fixed point (of multiplier $\lambda$ or $\approx 1/\lambda$) for all $t$ large.  Both critical points will lie in its basin of attraction for all $t$ large, demonstrating stability of the family $f_t$.  To see this, we may place the three fixed points of $f_t$ at $\{0,1,\infty\}$ so that $f_t$ is conjugate to the rational function
	$$g_t(z) = z \frac{(1-\lambda)z + \lambda(1-\beta)}{\beta(1-\lambda)z + 1-\beta}.$$
where $\beta = \beta(t) = \beta(-t) \approx 1/\lambda$ is the multiplier of the fixed point at $\infty$.  In this form, $g_t$ will converge (locally uniformly on $\mathbb{\widehat{C}}\setminus\{1\}$) to the linear map $z\mapsto \lambda z$ as $t\to \infty$.   In particular, there is a neighborhood $U$ containing the attracting fixed point and the point $z = \lambda$ mapped compactly inside itself by $g_t$ for all $t$ large.  On the other hand, we can explicitly compute the critical values $v_+(t), v_-(t)$ of $g_t$ and determine that $\lim_{t \to\infty} v_\pm(t) = \lambda$.  Consequently, the critical points lie in the basin of attraction for all $t$ large enough, and the bifurcation locus must be bounded.

For $\lambda = 1$, it is convenient to conjugate $f_t$ by $1/(tz)$, to express it in the form
	$$g_t(z) = z + 1 + \frac{1}{t^2 z}$$
with a parabolic fixed point at $z=\infty$.  In these coordinates, $g_t$ converges (locally uniformly on $\mathbb{\widehat{C}}\setminus\{0\}$) to the translation $z\mapsto z+1$ as $t\to \infty$.  Again, we compute explicitly the critical values of $g_t$ and their limit as $t\to \infty$; in this case, they converge to the point $z=1$.  As such, they both lie in the basin of the parabolic fixed point for $t$ large.
\qed

\subsection{Comparing the two bifurcation loci.}   We begin with a simple observation.

\begin{lemma}  \label{bifset}
For $\Re \lambda>1$, then $\Bif^+$ and $\Bif^-$ intersect only at the two points $t = \pm 2\sqrt{1-\lambda}$.    For $|\lambda|<1$, the sets $\Bif^+$ and $\Bif^-$ are disjoint.
\end{lemma}

\begin{figure} [b]
\includegraphics[width=2.05in]{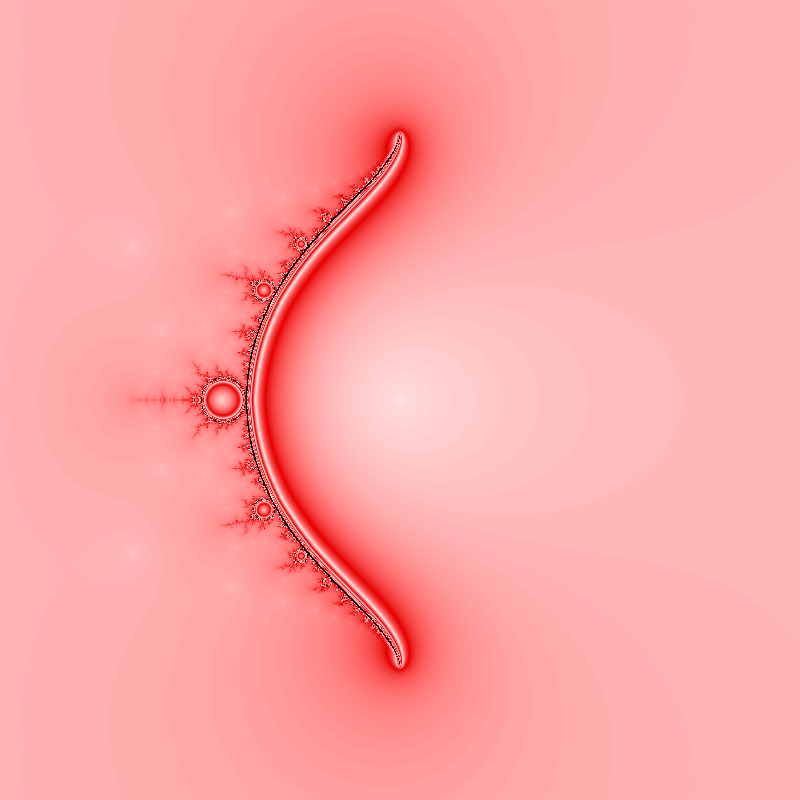}
\includegraphics[width=2.05in]{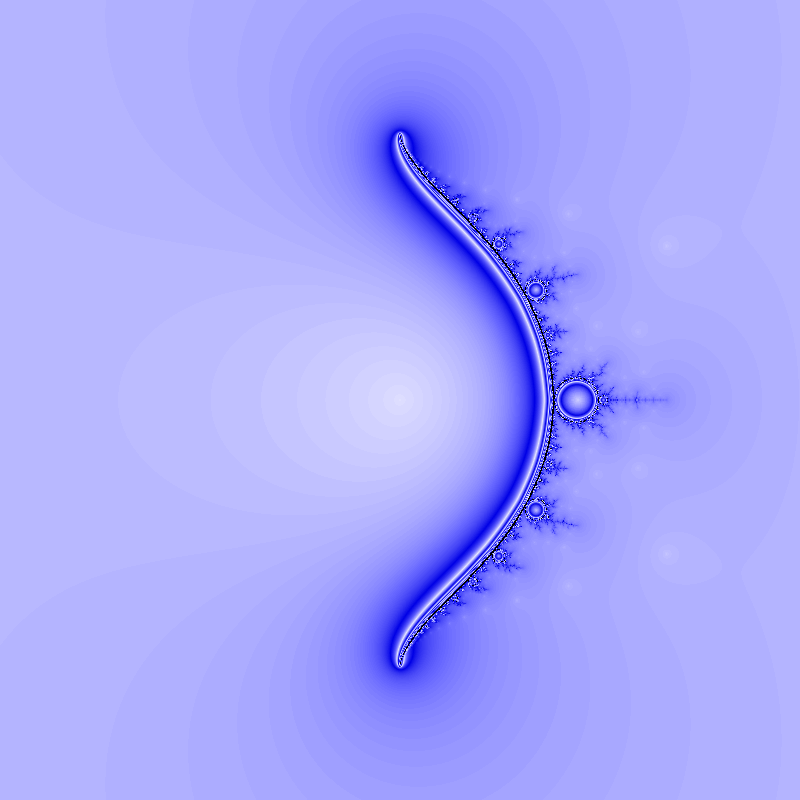}
\includegraphics[width=2.05in]{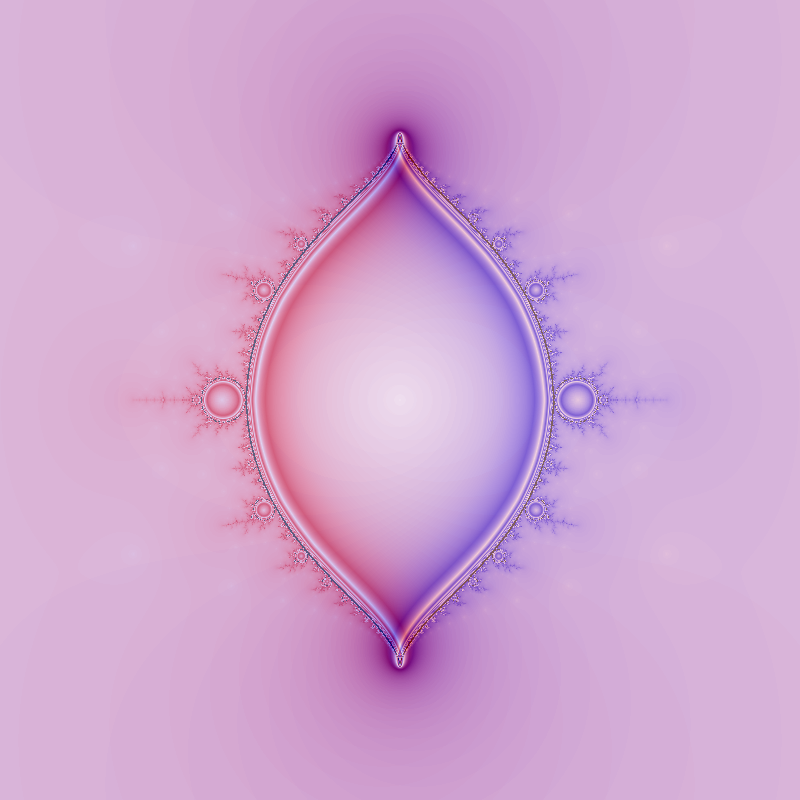}

\caption{ \small The bifurcation loci in $\Per_1(2)^{cm}$.  At left, an illustration of $\Bif^+$ where $|\Re t\, |, |\Im t\,| \leq 3$; the color shading records a rate of convergence of the critical point $c_+$ to an attracting cycle.   In the middle, $\Bif^-$ in the same region.  At right, the two images superimposed. By Lemma \ref{bifset}, $\Bif^+\cap\Bif^- = \{2i, -2i\}$.}
\label{L=2}
\end{figure}

 \proof
For $\Re \lambda > 1$ and $t\not= \pm 2\sqrt{1-\lambda}$, $f_t$ has at least one attracting fixed point.  The proof is immediate from the index formula for fixed point multipliers  (see \cite{Milnor:dynamics}).  For all such $t$ then, at least one of the critical points must be stable.  This shows that $\Bif^+\cap \Bif^- \subset \{\pm 2\sqrt{1-\lambda}\}$.  A straightforward calculation shows that the fixed point multipliers at $t=0$ are $\{\lambda, -1 + 2/\lambda, -1 + 2/\lambda\}$, so $f_0$ must have two distinct attracting fixed points whenever $\Re\lambda>1$.  Consequently, $t=0$ cannot lie in the unbounded stable component (where there is a unique attracting fixed point by Proposition \ref{compactness}).  For topological reasons, then, and since $\Bif^+ = -\Bif^-$, the intersection of $\Bif^+$ and $\Bif^-$ must consist of at least two points, concluding the proof that $\Bif^+ \cap \Bif^- = \{\pm 2\sqrt{1-\lambda}\}$.

For $|\lambda|<1$, the point $0$ is an attracting fixed point for all $t$, and its immediate attracting basin must contain at least one critical point.  Thus, for all $t$, at least one critical point remains in an attracting basin under perturbation, and so it is stable; this implies that $\Bif^+\cap \Bif^-=\emptyset$.
\qed

\bigskip
For certain values of $\lambda$, the bifurcation loci $\Bif^\pm$ are remarkably similar.  Though $\Bif^+$ does not appear to be equal to $\Bif^-$ for any value of $\lambda$, the differences can be subtle.  We include illustrations of $\Bif^\pm$ in $\Per_1(\lambda)^{cm}$ for three values of $\lambda$ in Figures \ref{L=2}-\ref{L=-4}.  In Figure \ref{L=-4 measures}, we illustrate the distribution of the parameters where the critical points are periodic, with $\lambda = -4$; these parameters converge to the bifurcation measures by Theorem \ref{equidistribution}.  Theorem \ref{distinct measures} states that the two measures $\mu^+_\lambda$ and $\mu^-_\lambda$ are distinct for all $\lambda$.

\begin{question} \label{distinct supports}
Are the bifurcation loci $\Bif^+$ and $\Bif^-$ distinct in $\Per_1(\lambda)^{cm}$ for all $\lambda$?  What explains their near-coincidence for parameters such as $\lambda = -4$?
\end{question}

\begin{figure} [h]
\includegraphics[width=2.05in]{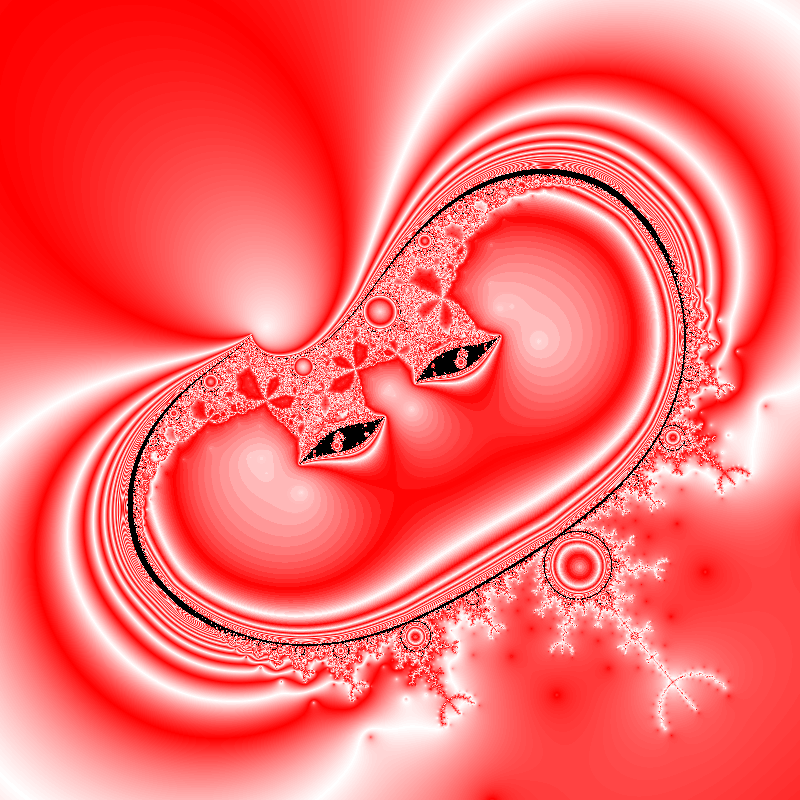}
\includegraphics[width=2.05in]{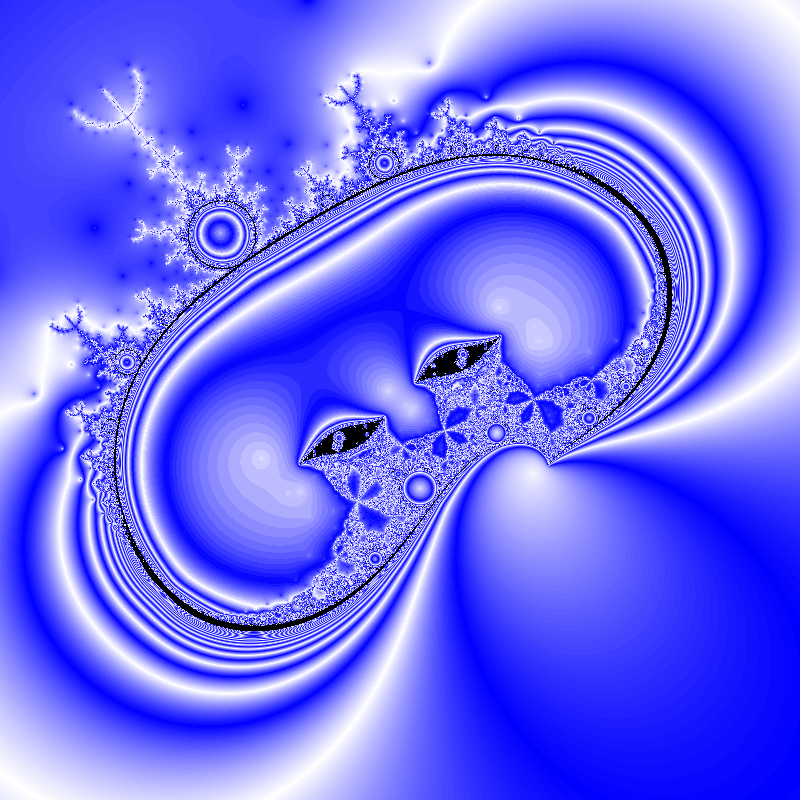}
\includegraphics[width=2.05in]{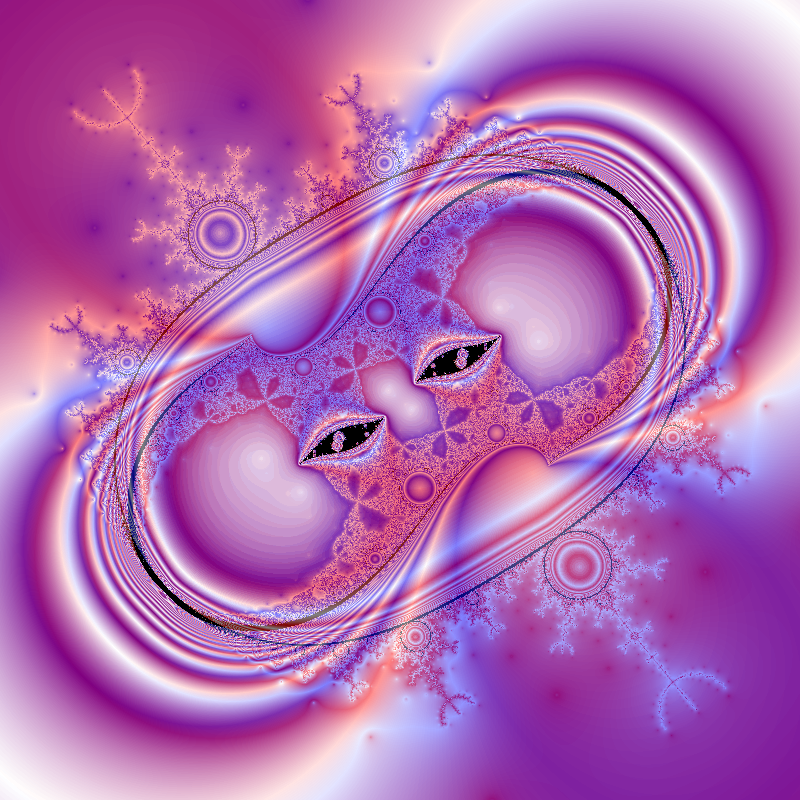}

\caption{ \small The bifurcation loci in $\Per_1(1.1i)^{cm}$.  At left, an illustration of $\Bif^+$ where $|\Re t\, |, |\Im t\,| \leq 6$; the color shading records a rate of convergence of the critical point $c_+$ to an attracting cycle.   In the middle, $\Bif^-$ in the same region.    At right, the two images superimposed. }
\label{L=1.1i}
\end{figure}

\begin{figure} [h]
\includegraphics[width=2.05in]{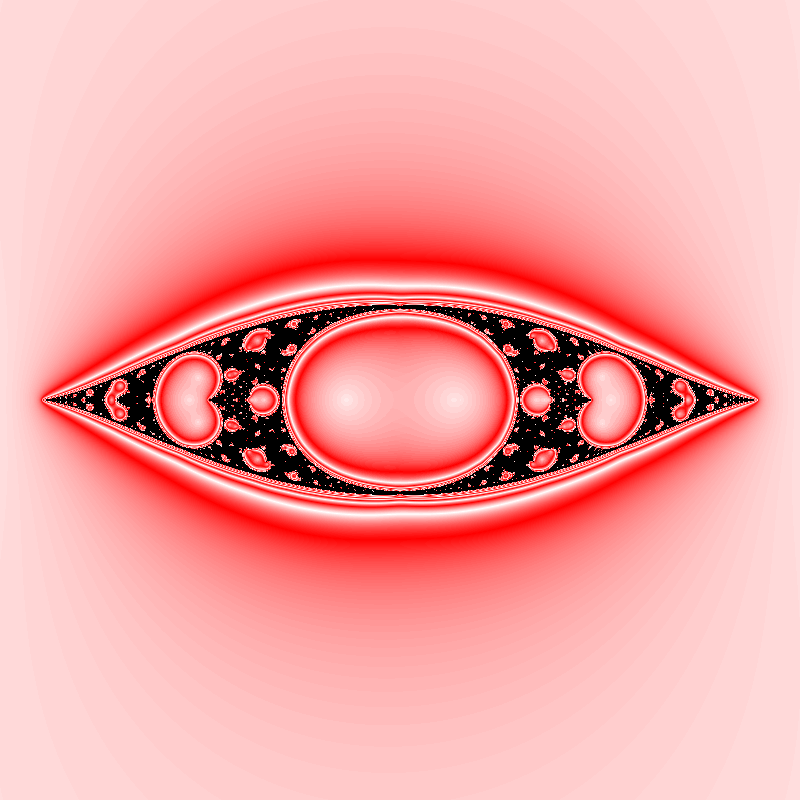}
\includegraphics[width=2.05in]{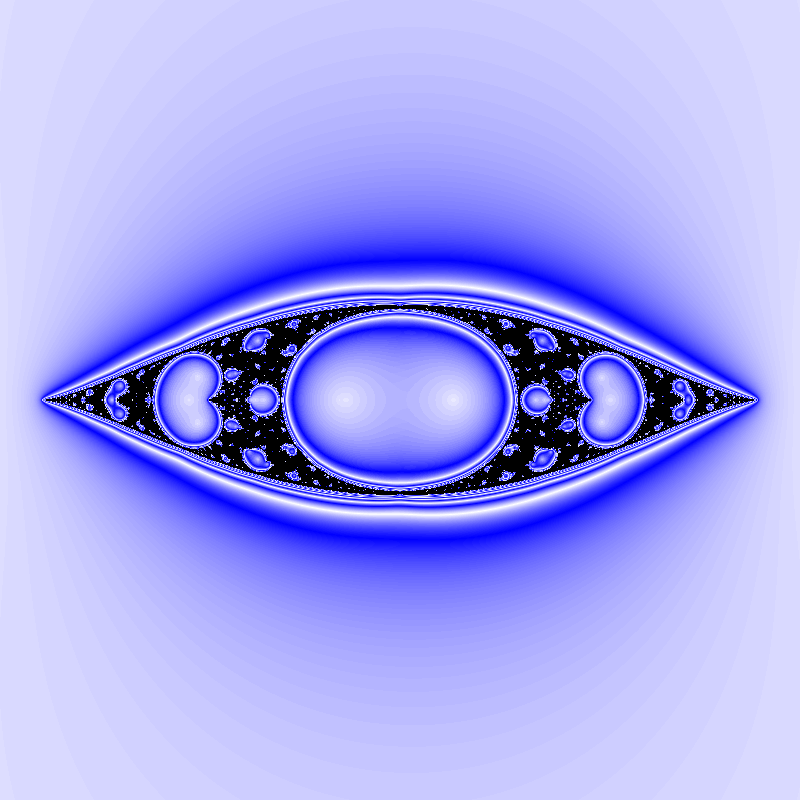}
\includegraphics[width=2.05in]{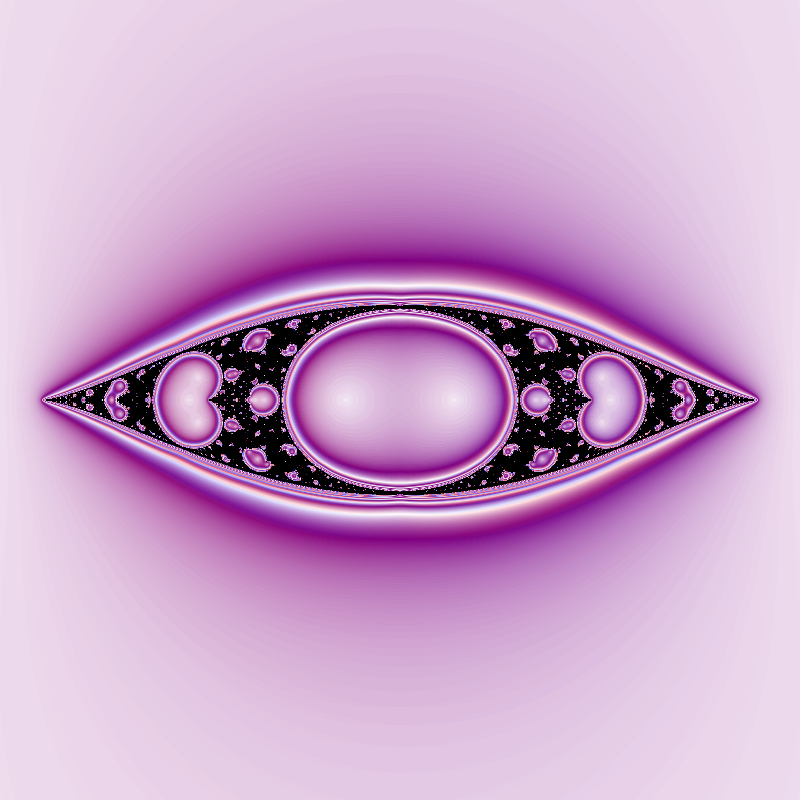}

\caption{ \small The bifurcation loci in $\Per_1(-4)^{cm}$.  At left, an illustration of $\Bif^+$ where $|\Re t\, |, |\Im t\,| \leq 5$; the color shading records a rate of convergence of the critical point $c_+$ to an attracting cycle.   In the middle, $\Bif^-$ in the same region.   At right, the two images superimposed.}
\label{L=-4}
\end{figure}

\begin{figure} [h]
\includegraphics[width=2.7in]{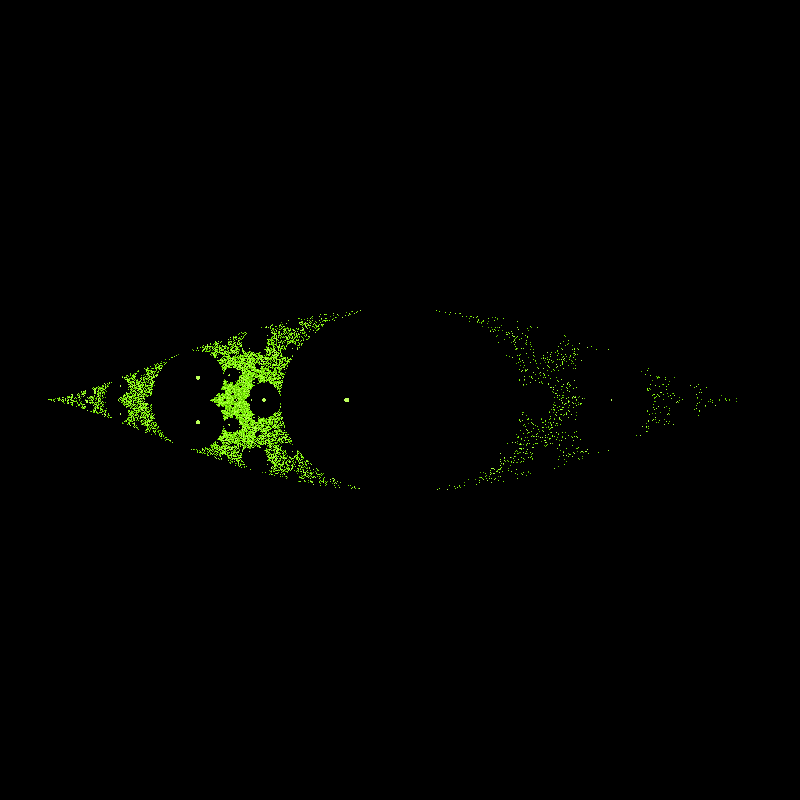}
\includegraphics[width=2.7in]{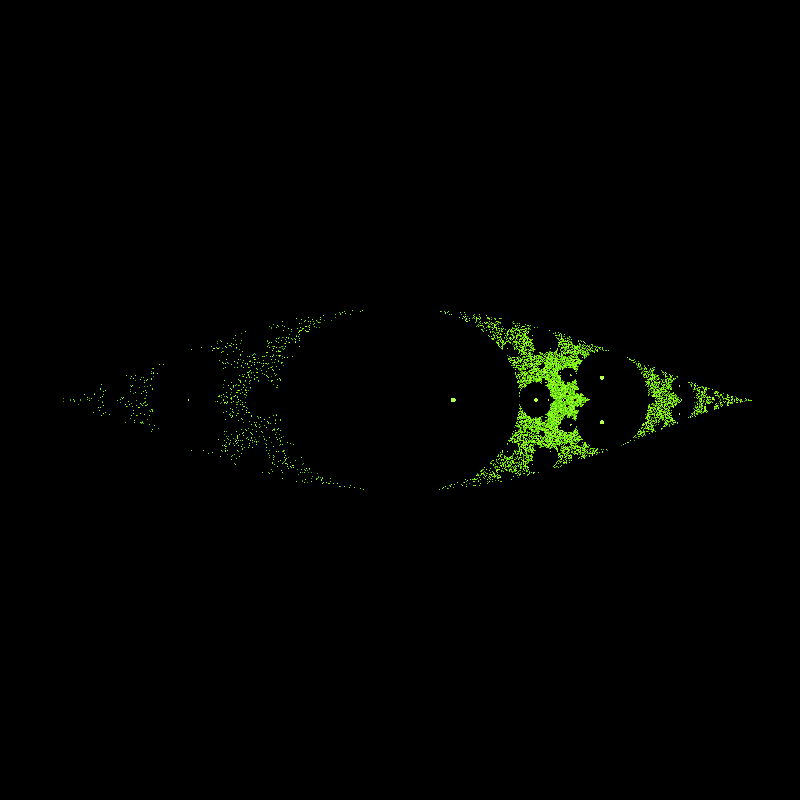}

\caption{ \small At left, a plot of parameters $t$ such that $f_t^n(+1) = +1$ in $\Per_1(-4)^{cm}$, with $n\leq 5000$.  By Theorem \ref{equidistribution}, these parameters are equidistributed with respect to $\mu^+_\lambda$ as $n\to \infty$.  At right, the parameters where critical point $-1$ is periodic, equidistributed with respect to $\mu^-_\lambda$.  }
\label{L=-4 measures}
\end{figure}

\subsection{Dynamical independence of the critical points.}  We conclude this section with the observation that the two critical points $c_+ = +1$ and $c_- = -1$ must be dynamically independent along $\Per_1(\lambda)^{cm}$.  We define,
	$$\Preper^\pm_\lambda = \{t\in\Per_1(\lambda)^{cm}:  \pm 1 \mbox{ has finite forward orbit for } f_t\}.$$

\begin{prop}  \label{no synchrony}
For all $\lambda\in\C$, we have
	$$\Preper^+_\lambda \not= \Preper^-_\lambda$$
in $\Per_1(\lambda)^{cm}$.
\end{prop}

\proof
The case of $\lambda=0$ is easy.  The curve $\Per_1(0)^{cm}$ has two irreducible components; each may be parameterized by $f_t(z) = z^2 + t$ with one critical point at $\infty$ and the other at 0.  The critical point at $\infty$ is fixed for all $t$, while the orbit of $0$ is infinite for all but countably many parameters $t$.

For $0 < |\lambda|\leq 1$, a stronger statement is true:
	$$\Preper^+_\lambda \cap \Preper^-_\lambda = \emptyset$$	
in $\Per_1(\lambda)^{cm}$.  Indeed, for every $f \in \Per_1(\lambda)^{cm}$, at least one critical point must have infinite forward orbit, as it is attracted to (or accumulates upon) the fixed point with multiplier $\lambda$ (or on the boundary of the Siegel disk in case the fixed point is of Siegel type).  See, for example, \cite[Corollaries 14.4 and 14.5]{Milnor:dynamics}.

For the remainder of this proof, assume that $|\lambda|>1$ and that $\Preper^+_\lambda = \Preper^-_\lambda$.  Lemma \ref{activity} shows that $\Bif^+$ is contained in the set of accumulation points of $\Preper^+_\lambda$.  But, in fact, the characterizations of stability (as in \cite[Chapter 4]{McMullen:CDR}) imply that elements of $\Preper^+_\lambda$ cannot accumulate in a stable region.  So we have $\Bif^+ = \Bif^-$.  

As before, parameterize $\Per_1(\lambda)^{cm}$ as
	$$f_t(z) = \frac{\lambda z}{z^2 + tz + 1}.$$
Recall from Proposition \ref{compactness} that $f_t$ has an attracting fixed point for all $t$ near $\infty$, with multiplier converging to $1/\lambda$ as $t\to \infty$.   Moreover, both critical points lie in its basin of attraction for all $t$ in the unbounded stable component.  In particular, each $f_t$ in the unbounded stable component has a unique attracting fixed point.

The multiplier of the unique attracting fixed point of $f_t$ defines a holomorphic function from the unbounded stable component to the unit disk.  Recall that the fixed-point multiplier cannot be equal to $1/\lambda$ for any $f_t$, and it will converge to $1/\lambda$ if and only if $t\to\infty$ in $\Per_1(\lambda)^{cm}$  \cite{Milnor:quad}.   Moreover, since $\Bif^+ = \Bif^-$, the multiplier must converge to 1 in absolute value as $t$ tends to the bifurcation locus.  Consequently, the multiplier of the attracting fixed point determines a proper holomorphic map from the unbounded stable component to the unit disk punctured at $1/\lambda$.  It follows that each preimage of the line segment $[0, 1/\lambda)$ defines a path from a parameter $t_0$ to infinity in this unbounded stable component.  In fact, $t_0 = \lambda - 2$ is the unique parameter where the critical point $+ 1$ is fixed (and similarly, $2-\lambda$ is the unique parameter at which $-1$ is fixed), so $f_{t_0}$ is conjugate to a polynomial; we have just shown that $t_0$ lies in this unbounded stable component. 

As the orbit of the critical point $-1$ for $f_{t_0}$ must converge towards the fixed critical point $+1$, we see that $-1$ has infinite orbit.  This contradicts our assumption that $\Preper^+_\lambda = \Preper^-_\lambda$ and completes the proof.
\qed

\bigskip
As an immediate corollary of Proposition \ref{no synchrony}, we see that the two critical points are dynamically independent on $\Per_1^{cm}(\lambda)$ for any $\lambda\not=0$, in the sense of critical orbit relations as formulated in \cite[Question 6.4]{D:stableheight} (see also \cite[\S1.4]{BD:polyPCF}).  

\begin{cor}  \label{independence}
Fix $\lambda\in\C^*$.  The critical points $c_+ = +1$ and $c_- = -1$ cannot satisfy any dynamical relations along $\Per_1^{cm}(\lambda)$.  In particular, for each pair of integers $n, m\geq 0$, there exists $f_t\in \Per_1(\lambda)^{cm}$ so that $f_t^n(c_+) \not= f_t^m(c_-)$.
\end{cor}

\proof
Dynamical dependence of points $c_+$ and $c_-$, in the sense of \cite[Question 6.4]{D:stableheight}, means that there exist rational functions $A_t$ and $B_t$, commuting with $f_t$ for all $t\in \Per_1^{cm}(\lambda)$, so that $A_t(c_+) = B_t(c_-)$ for all $t$.  This includes orbit relations such as $f_t^n(c_+) = f_t^m(c_-)$ for all $t$.  Dependence implies that $c_+$ is preperiodic for $f_t$ if and only if $c_-$ is preperiodic for $f_t$.  This contradicts Proposition \ref{no synchrony}.
\qed

\begin{remark}
In contrast with Corollary \ref{independence}, conditions on the multipliers can and do impose relations between the critical points in other settings.  For example, if we look at conjugacy classes $f\in M_2$ with two distinct period-3 cycles of the same multiplier, then we obtain the automorphism locus $\mathcal{A}_2$ \cite[Theorem 3.1]{Berker:Epstein:Pilgrim}.  The family $\mathcal{A}_2$ is given by $f_{\lambda,0}(z) = \lambda z /(z^2 + 1)$, for parameter $\lambda\in\C^*$, with the automorphism $A(z) = -z$ for all $\lambda$; the critical points (at $\pm 1$) and their orbits are symmetric by $A$, and thus they define the same bifurcation locus and equal bifurcation measures.
\end{remark}

\bigskip
\section{The bifurcation measures are distinct}
\label{measures}

In this section, we provide a proof of Theorem \ref{distinct measures}.  Recall the definitions from \S \ref{definitions}.

\subsection{Potential functions for the bifurcation measures.}
Recall the definitions of the measures $\mu^{\pm}_\lambda$ and their potential functions $H_\lambda^{\pm}$ from \S\ref{definitions}.  We begin by showing that if the two bifurcation measures were to coincide, their potential functions would have to be equal.  The first lemma controls the growth of $H_\lambda^{\pm}$.  The lower bound will be used again in the proof of Theorem \ref{equidistribution}.  We will work with the norm
	$$\|(z_1, z_2)\| = \max\{|z_1|, |z_2|\}$$
on $\C^2$.

\begin{lemma}  \label{H bound}
For each $\lambda\neq 0$, there are constants $c, C >0$, such that
	 $$c |t|^{-1} \leq \frac {\| F_t (z_1,z_2) \|}{\|(z_1,z_2)\|^2} \leq C |t|$$
for all $|t|\geq 1$ and all $(z_1, z_2)\neq (0,0)$.   Consequently,
	$$H_\lambda^\pm(t) = O(\log |t|)$$
as $t\to \infty$.
\end{lemma}

\proof
The upper bound is immediate from the expression $F_t(z_1, z_2) =  (\lambda z_1 z_2, z_1^2 + t z_1 z_2 +  z_2^2)$.  We may set
	$$C = \max\{|\lambda|, 3\}.$$
For the lower bound, by the symmetry and homogeneity of $F_t$, we may assume that $z_2=1$ and $|z_1|\leq 1$.  Then $\|(z_1,z_2)\| = 1$, and we shall estimate the norm of $sF_{1/s}(z_1,1) = (s\lambda z_1, sz_1^2 + s + z_1)$ with $|s| \leq 1$.  Let
	$$c = \min\{|\lambda|/2, 1/4\}.$$
For each $s$ with $|s| \leq 1$, either $|s\lambda z_1| \geq c|s|^2$, or $|z_1| < |s|/2$ in which case,
   $$|sz_1^2+s+z_1|  \geq |s| - |s|/2  -  |s|^3/4 \geq |s|/4  \geq c |s|^2.  $$
Consequently, $\|F_{1/s}(z_1,z_2)\| \geq c|s|$ and the lower bound is proved.

By the identity
$$H_\lambda^+(t)=\sum_{i=2}^{+\infty}\frac{1}{2^i}\log\Bigg(\frac{\|F_t^i(1,1)\|}{\|F_t^{i-1}(1,1)\|^2}\Bigg)
+\frac{1}{2}\log\|F_t(1,1)\|,$$
we have that $|H_\lambda^+(t)|<3\log|t|$ when $t$ is large.  The same holds for $H_\lambda^-$, since $H_\lambda^-(t) = H_\lambda^+(-t)$.
\qed

\begin{lemma}\label{H1=H2}
For any $\lambda\neq0$,  we have
$$\mu^+_\lambda=\mu^-_\lambda \Longrightarrow H_\lambda^+=H_\lambda^-.$$
\end{lemma}

\proof
Let $h(t)=H_\lambda^+(t) - H_\lambda^-(t)$.  If $\mu^+_\lambda = \mu^-_\lambda$, then $\Delta h=0$ in $\mathbb{C}$.  This implies that $h$ is harmonic.  By Lemma \ref{H bound}, we have $h(t)=O(\log|t|)$ for $t$ near $\infty$.  Therefore $h$ is constant.  Combined with the symmetry $H_\lambda^-(t) = H_\lambda^+(-t)$, we may conclude that $H_\lambda^+ = H_\lambda^-$.
\qed

\subsection{Showing that the potentials differ at a single point}
As a consequence of Lemma \ref{H1=H2}, it suffices to show that $H^+$ and $H^-$ differ at a single point.  For $t=\lambda-2$, we may compute the values.  Observe that $f_{\lambda-2}$ is conjugate to a polynomial, as $f_{\lambda-2}(1) = 1$.

\begin{lemma}\label{H12} We have
$$H_\lambda^+(\lambda-2)=\log|\lambda| \quad \mbox{ and } \quad  \ H_\lambda^-(\lambda-2)=\frac{1}{2}G_{\mathcal{M}}(c(\lambda))+\log 2,$$
when $G_{\mathcal{M}}$ is defined in Example \ref{quadratic G} and $c(\lambda)=\frac{\lambda}{2}-\frac{\lambda^2}{4}$.
\end{lemma}

 \proof  The map $F_t$ defined in (\ref{F_t}) for $t=\lambda-2$ satisfies $F_{\lambda-2}(1/\lambda, 1/\lambda)=(1/\lambda,1/\lambda)$.  Therefore,
   $$H_\lambda^+(\lambda-2)= \lim_{n\to\infty} \frac{1}{2^n} \log\|F_{\lambda-2}^n(1/\lambda, 1/\lambda)\| + \log|\lambda|=\log|\lambda|.$$

For $t=\lambda-2$, recall that $f_{\lambda-2}$ is conjugate to a polynomial
\begin{equation}\label{quadratic poly}
q_\lambda(z) = \lambda z(z+1).
\end{equation}
Set
$$A(z_1,z_2)=(z_1,z_2-z_1) \quad \mbox{ and } \quad \tilde{F}_{\lambda,t}=A\circ F_{\lambda,t}\circ A^{-1}.$$
When $t=\lambda-2$, we have
$$\tilde{F}_{\lambda,\lambda-2}(z_1, z_2)=(\lambda z_1(z_1+z_2), z_2^2).$$
Note that $A(-1,1)=2(-1/2,1), \  \tilde{F}_{\lambda, \lambda-2}^n(A(-1,1))=2^{2^n}(q_\lambda^n(-1/2),1)$,
so that
\begin{equation}\label{H2-}
\begin{array}{lll}
H_\lambda^-(\lambda-2)
&=\lim_{n\rightarrow\infty}2^{-n}\log\|F_{\lambda,\lambda-2}^n(-1,1)\|\\[6pt]
&=\lim_{n\rightarrow\infty}2^{-n}\log\|A^{-1}\circ \tilde{F}_{\lambda,\lambda-2}^n(A(-1,1))\|\\[6pt]
&=\lim_{n\rightarrow\infty}2^{-n}\log\|\tilde{F}_{\lambda,\lambda-2}^n(A(-1,1))\|\\[6pt]
&=\lim_{n\rightarrow\infty}2^{-n}\log^+|q_\lambda^n(-1/2)|+\log 2.
\end{array}
\end{equation}
In this way, we express $H^-_\lambda(\lambda-2)$ in terms of
 the escape rate of the critical point $-1/2$ of the polynomial $q_\lambda$. We may conjugate $q_\lambda$ to the quadratic polynomial
$$p_{c(\lambda)}(z) = z^2+c(\lambda), \qquad  \mbox{with } c(\lambda)=\frac{\lambda}{2}-\frac{\lambda^2}{4}.$$
Then we get
$$H^-_{\lambda}(\lambda-2)=\frac{1}{2}G_{\mathcal{M}}(c(\lambda))+\log 2.$$
\qed

Now we are ready to compare $H^+_\lambda(\lambda-2)$ and $H^-_\lambda(\lambda-2)$.

\begin{lemma}\label{H1<=H2}
For any $\lambda\neq 0$ and $\Re \lambda\leq 1$, we have
\begin{equation}\label{H1 leq H2}
H_\lambda^+(\lambda-2)\leq H_\lambda^-(\lambda-2),
\end{equation}
with equality if and only if $\lambda=-2$.
\end{lemma}

\proof By Lemma \ref{H12}, we need to show that
\begin{equation}
G_{\mathcal{M}}(c(\lambda))\geq 2\log|\lambda/2|,
\end{equation}
for $\lambda\neq 0$ and $\Re \lambda\leq 1$.
The proof relies on the following two claims:

\medskip\noindent
{\it Claim 1: If $\Re \lambda=1$, we have
\begin{equation}\label{hexi111}
G_{\mathcal{M}}(c(\lambda))> 2\log|\lambda/2|.
\end{equation}}

\noindent
{\it Claim 2: For $|\lambda|=2$, $\Re \lambda<1$, the parameter $c(\lambda)=\frac{\lambda}{2}-\frac{\lambda^2}{4}$ lies in the Mandelbrot set $\mathcal{M}$ if and only if $\lambda=-2$.}

Lemma \ref{H1<=H2} follows easily from Claims 1 and 2.  Indeed, for $|\lambda|<2$, we have $G_{\mathcal{M}}(c(\lambda))\geq 0>2\log|\lambda/2|$.  By Claim 2, for $|\lambda|=2$, $\Re \lambda<1$ and $\lambda\neq-2$,  we have $G_{\mathcal{M}}(c(\lambda))>0= 2\log|\lambda/2|$, and for $\lambda=-2$, $G_{\mathcal{M}}(c(\lambda))=G_{\mathcal{M}}(-2)=0= 2\log|\lambda/2|$.  By Claim 1, if $\Re \lambda=1$, we have
$ G_{\mathcal{M}}(c(\lambda))> 2\log|\lambda/2|$.   It follows from Claim 2 that $G_{\mathcal{M}}(c(\lambda))$ is harmonic in the region $\Omega=\{\Re \lambda<1, |\lambda|>2\}$. 
Observe that when $\lambda\rightarrow\infty$ in $\Omega$,
$$G_{\mathcal{M}}(c(\lambda))-2\log|\lambda/2|\rightarrow 0.$$
By the maximum/minimum value theorem,  we have $ G_{\mathcal{M}}(c(\lambda))> 2\log|\lambda/2|$ in $\Omega$. The conclusion then follows.

\medskip
\noindent{\it Proof of Claim 1.} For $\Re \lambda=1$, we have $2-\lambda=\overline{\lambda}$ and $c(\lambda)=|\lambda|^2/4$.  It is equivalent to show that $G_\mathcal{M}(c)>\log c$ when $c>1/4$ and $c\in\mathbb{R}$. For $p_c(z)=z^2+c$ with $c>1/4$, we have $p_c(c)=c^2+c$ and  $p_c^{n}(c)\geq (c^2+c)^{2^{n-1}}$.
 Consequently,
$$G_{\mathcal{M}}(c)\geq \lim_{n\rightarrow\infty} 2^{-n}\log (c^2+c)^{2^{n-1}}=\frac{1}{2}\log(c^2+c)>\log c.$$

\medskip
\noindent{\it Proof of Claim 2.}
Let $p_c(z)=z^2+c$. Recall that the Mandelbrot set $\mathcal{M}$ can be defined by
$$\mathcal{M} = \{c\in \C: |p_c^n(0)|\leq 2 \textup{ for any $n\geq 1$}\}.$$
In order to show $c(\lambda)\notin \mathcal{M}$ when $|\lambda|=2$, $\Re \lambda<1$, by the above definition, it suffices to show that \begin{equation}\label{mandelbrot}
|p_{c(\lambda)}^2(0)|=|(c(\lambda))(c(\lambda)+1)|>2.
\end{equation}
Let $\lambda=2(\cos\theta+i\sin\theta)$ with $\theta \in [\pi/3, 5\pi/3]$. Then $\cos\theta\in [-1, 1/2]$. With some computation,  one has
$$|(c(\lambda))(c(\lambda)+1)|=\sqrt{2(5-5\cos\theta-4\cos^2\theta+4\cos^3\theta)}$$
 Let $u=\cos\theta\in [-1, 1/2]$,  the function $g(u)=5-5u-4u^2+4u^3$ have minimum when $u=-1$ or $1/2$.
   Consequently, for any $\lambda$ with  $|\lambda|=2$, $\Re \lambda<1$ and $\lambda\neq-2$, the inequality (\ref{mandelbrot}) holds; i.e., $c(\lambda)$ is not in the Mandelbrot set.
\qed

\medskip
Finally, we treat the case of $\lambda = -2$.

\begin{prop}\label{QPer1(-2)}
For $\lambda=-2$, the bifurcation sets $\Bif^+$ and $\Bif^-$ are not equal.
\end{prop}

\begin{proof}
For $\lambda=-2$ and $t=2\sqrt{3}$, the map $f_{\lambda,t}=f_{-2,2\sqrt{3}}$ has a fixed point at $z=-\sqrt{3}$ with multiplier $1$.  For all $t> 2\sqrt{3}$, one of the fixed points is attracting.  Using Proposition \ref{compactness}, to show the bifurcation sets are different, it suffices to show the existence of $t_0>2\sqrt{3}$ so that $f_{-2, t_0}$ has a second attracting cycle of period $>1$.  In that case, both of the critical points cannot lie in the basin of the attracting fixed point.

For $t> 2\sqrt{3}$, consider the first three iterations of $1$ under $f_{-2,t}$
$$1\mapsto \frac{-2}{2+t}\mapsto\frac{8+4t}{8-t^2}\mapsto f_{-2,t}^3(1)=\frac{-8(2+t)(8-t^2)}{(8-t^2)^2+4t(t+2)(8-t^2)+16(2+t)^2}.$$
To see $f_{-2, t}^3(1)=1$ has a solution for $t>2\sqrt{3}$, define
\begin{eqnarray*}
\ell(t)&=&(8-t^2)^2+4t(t+2)(8-t^2)+16(2+t)^2+8(2+t)(8-t^2)\\
&=&16(2+t)^2+4(2+t)^2(8-t^2)+(8-t^2)^2
\end{eqnarray*}
Note that
$$\ell(2\sqrt{3})=16>0;  \lim_{t\rightarrow+\infty}\ell(t)=-\infty,$$
this implies that $\ell(t)=0$ has a solution for some $t_0>2\sqrt{3}$. Since $f_{-2, t_0}(1)\neq1$, the critical point  $1$ is of periodic $3$ for $f_{-2,t_0}$.
 \end{proof}

\subsection{Proof of Theorem \ref{distinct measures}.} First, suppose $|\lambda|<1$ or $\Re \lambda>1$. By Lemma \ref{bifset}, the two bifurcation sets $\Bif^+$ and $\Bif^-$ are not equal.  As the supports of $\mu^+_\lambda$ and $\mu^-_\lambda$ are exactly $\Bif^+$ and $\Bif^-$, it follows that $\mu^+_\lambda\neq \mu^-_\lambda$.

Second, if $\lambda=-2$, by Proposition \ref{QPer1(-2)}, we know $\Bif^+ \neq \Bif^-$ which implies  $\mu^+_\lambda \neq \mu^-_\lambda$.

Finally, assume $\Re \lambda\leq 1$ and $\lambda\neq 0, -2$ and suppose that  $\mu^+_\lambda=\mu^-_\lambda$.  Then by  Lemma \ref{H1=H2}, we have  $H_\lambda^+(\lambda-2)=H_\lambda^-(\lambda-2)$.  However, this contradicts  Lemma \ref{H1<=H2}.
\qed

\bigskip
\section{Homogeneous potential functions}
\label{homogeneous}

In this section, we study the potential functions $H_\lambda^+(t)$ and $H_\lambda^-(t)$ of (\ref{H}) in more detail.  From Lemma \ref{H bound} we know that $H_\lambda^\pm(t) = O(\log|t|)$ as $t\to \infty$.  Here, we refine this estimate and prove Theorem \ref{convergence}, the first step in our proof of Theorem \ref{equidistribution}.

\subsection{The homogeneous potential functions on parameter space.}
\label{F_n}
Fix $\lambda\not=0$.  Working in homogeneous coordinates, we write
	$$F_t(z_1,z_2) = (\lambda z_1 z_2, z_1^2 + t z_1 z_2 +  z_2^2)$$
for a lift of $f_{\lambda,t}$ to $\C^2$, with $z = z_1/z_2$.  We will also work in homogeneous coordinates over the parameter space. Consider the two sequences of maps,
\begin{equation}  \label{F_n plus}
	F_n^+(t_1,t_2) := t_2^{2^{n-1}} F^n_{t_1/t_2}(+ 1, 1),
\end{equation}
and
\begin{equation}  \label{F_n minus}
	F_n^-(t_1,t_2) := t_2^{2^{n-1}} F^n_{t_1/t_2}(-1, 1),
\end{equation}
for $(t_1, t_2) \in \C^2$.

\begin{theorem}  \label{convergence}
The maps $F_n^\pm$ are homogeneous polynomial maps in $(t_1,t_2)$ of degree $2^{n-1}$ with nonzero resultants.  For each $\lambda\not=0$ such that
	$$\gamma(\lambda)=\frac{1}{2}\sum_{i=1}^{+\infty}\frac{1}{2^i}\log|1+\lambda+\cdots+\lambda^{i}|$$
converges, the limits
	$$\lim_{n\to\infty} \frac{1}{2^{n-1}} \log \| F_n^\pm(t_1,t_2) \|$$
converge locally uniformly on $\C^2\setminus\{(0,0)\}$ to continuous functions $G^\pm$ satisfying
	$$G^\pm(t_1,t_2) = \begin{cases}
 				2H^\pm_\lambda(t_1/t_2)+\log |t_2| &\text{ if } t_2\neq0,\\
				\log|t_1|+ \gamma(\lambda) &\text{ if } t_2=0.
			\end{cases}$$
\end{theorem}

\smallskip
\begin{remark}
It is easy to see that $\gamma(\lambda)$ is finite for all $\lambda\in\C$ with $|\lambda|\not=1$ and for $\lambda=1$.  In the next section, we observe that it is finite for all algebraic numbers $\lambda$ that are not roots of unity.
\end{remark}

For the proof of Theorem \ref{convergence}, it suffices to consider the maps $F_n^+$ and the function $G^+$; the results for $F_n^-$ and $G^-$ follow by symmetry.  Define polynomials $P_n(t)$ and $Q_n(t)$ by
	$$F_{t}^n(1,1) = (P_n(t), Q_n(t)).$$
One may verify by induction that the degree of $P_n$ is $2^{n-1}-1$, and the degree of $Q_n$ is $2^{n-1}$.  This shows that $F_n^+$ is polynomial in $(t_1,t_2)$.  Also, since $F_t^{-1}\{(0,0)\} = \{(0,0)\}$ for all $t\in\C$, we see that $P_n$ and $Q_n$ have no common roots.  Thus, $F_n^+$ has nonzero resultant in $(t_1,t_2)$.

For the convergence statement, note that standard arguments from complex dynamics imply that the convergence is uniform away from $t_2 = 0$.  In fact, the escape-rate function for $F_t$ will be continuous in both the dynamical variable $(z_1,z_2)$ and the parameter $t$ for any holomorphic family \cite{Hubbard:Papadopol, Fornaess:Sibony}.  It follows immediately from the definitions that
	$$G^+(t_1,t_1) = 2 H^+_\lambda(t_1/t_2) + \log|t_2|$$
whenever $t_2\not=0$.

The remainder of this section is devoted to the proof of uniform convergence near $t_2=0$.  The proof of Theorem \ref{convergence} will be complete once we have proved Lemmas \ref{upper bound by epsilon} and \ref{lower bound by epsilon} below.  We make an effort to include all details, especially the steps that will be repeated in the nonarchimedean setting in the following section.

\subsection{Convergence near $t_2=0$.}
Throughout this subsection, we set $t_1 =1$ and $t_2 = s$.  We have $$F_{n+1}^+(1, s) = F_{1/s} (F_n^+(1, s)).$$

We begin by looking at the coefficients of $F_n^+(1,s)$.  Write
\begin{equation} \label{coefficients}
   F_n^+(1,s)=(sB_n(\lambda)+s^2A_n(\lambda)+O(s^3), C_n(\lambda)+sD_n(\lambda)+O(s^2)).
\end{equation}
Note that $B_1(\lambda)=\lambda$, $C_1(\lambda)=1$, and
$$B_{n+1}(\lambda)=\lambda B_n(\lambda)C_n(\lambda) \quad\mbox{ and }\quad C_{n+1}(\lambda)=C_n(\lambda)(B_n(\lambda)+C_n(\lambda))$$
for all $n\geq 1$.  By induction, we obtain explicit expressions
\begin{eqnarray}
\label{B coeff}
B_n(\lambda)&=& \lambda^n(1+\lambda)^{2^{n-3}}(1+\lambda+\lambda^2)^{2^{n-4}}\cdots(1+\lambda+\cdot\cdot+\lambda^{n-2})^{2^0} \\
\label{C coeff}
C_n(\lambda) &=& B_n(\lambda)(1+\lambda+\cdot\cdot+\lambda^{n-1})/\lambda^n
\end{eqnarray}
for all $n\geq 3$.

\begin{lemma} \label{coefficients estimate}
The coefficients $A_n(\lambda)$, $B_n(\lambda)$, $C_n(\lambda)$, and $D_n(\lambda)$ of (\ref{coefficients}) satisfy
\begin{eqnarray*}
e^{\gamma(\lambda)}
&=&\lim_{n\to \infty} |B_n(\lambda)|^\frac{1}{2^{n-1}}=\lim_{n\to \infty} |C_n(\lambda)|^\frac{1}{2^{n-1}} \\
&\geq& \limsup_{n\to \infty} |A_n(\lambda)|^\frac{1}{2^{n-1}},  \limsup_{n\to \infty} |D_n(\lambda)|^\frac{1}{2^{n-1}}
\end{eqnarray*}
where $\gamma(\lambda)$ is defined in Theorem \ref{convergence}.
\end{lemma}

\proof
With the explicit expressions for $B_n$ and $C_n$ given above, the limiting value is clearly $e^{\gamma}$; it suffices to show the bound for $A_n$ and $D_n$.  By induction, we find
	$$A_{n+1}(\lambda)=\lambda \left(B_n(\lambda)D_n(\lambda)+C_n(\lambda)A_n(\lambda)\right),$$
and
	$$D_{n+1}(\lambda)=C_n(\lambda)(2D_n(\lambda)+A_n(\lambda))+D_n(\lambda)B_n(\lambda),$$
with $A_1(\lambda)=0$ and $D_1(\lambda)=2$.

The explicit expressions for $B_n(\lambda)$ and $C_n(\lambda)$ and these inductive formulas for $A_n(\lambda), D_n(\lambda)$, show that
$$A_{n}(\lambda)=(1+\lambda)^{2^{n-3}-2}(1+\lambda+\lambda^2)^{2^{n-4}-2}\cdot\cdot (1+\lambda+\cdots+\lambda^{n-3})^{0}A_{n}^*(\lambda),$$
and
$$D_{n}(\lambda)=(1+\lambda)^{2^{n-3}-2}(1+\lambda+\lambda^2)^{2^{n-4}-2}\cdot \cdot(1+\lambda+\cdots+\lambda^{n-3})^{0}D_{n}^*(\lambda),$$
for sequences $A_{n}^*(\lambda)$ and $D_{n}^*(\lambda)$ given inductively by
$$A_{n+1}^*(\lambda)=\lambda^{n+1} D_n^*(\lambda)+(\lambda+\lambda^2 +\cdot\cdot +\lambda^{n})A_n^*(\lambda)$$
and
$$D_{n+1}^*(\lambda)=(2(1+\lambda +\cdot\cdot +\lambda^{n-1})+\lambda^n)D_n^*(\lambda)+(1+\lambda +\cdot\cdot +\lambda^{n-1})A_n^*(\lambda)$$
with $A_1^*(\lambda)=0$ and $D_1^*(\lambda)=2$.  Finally, the induction formulas for $A_{n}^*(\lambda)$ and $D_{n}^*(\lambda)$ imply that
\begin{eqnarray*}
\max (|A_n^*(\lambda)|, |D_{n}^*(\lambda)|)&\leq& (2+2|\lambda|)^n \max (|A_{n-1}^*(\lambda)|, |D_{n-1}^*(\lambda)|)\\
                                                     &\leq& (2+2|\lambda|)^{2n-1} \max (|A_{n-2}^*(\lambda)|, |D_{n-2}^*(\lambda)|)\\
                                                     &\leq& 2(2+2|\lambda|)^{1+2+\cdots+n}
\end{eqnarray*}
Consequently, we have $ \lim_{n\to \infty} \sup |A_n(\lambda)|^\frac{1}{2^{n-1}}, \lim_{n\to \infty} \sup |D_n(\lambda)|^\frac{1}{2^{n-1}}\leq e^{\gamma(\lambda)}$.
\qed

\bigskip
The growth of the coefficients $B_n$ and $C_n$ in Lemma \ref{coefficients estimate} provides a uniform upper bound on the size of $2^{1-n} \log\|F_n^+(1,s)\|$ for small $s$:

\begin{lemma} \label{upper bound by epsilon}
For any given $\eps>0$, there exists a $\delta>0$ and an integer $N>0$ so that
	$$\frac{1}{2^{n-1}} \log \| F_n^+(1,s) \| - \gamma(\lambda) < \eps$$
for all $|s| < \delta$ and all $n\geq N$.
\end{lemma}

\proof
Define polynomials $p_n(s)$ and $q_n(s)$ by
   $$F_n^+(1,s) = (sp_n(s), q_n(s))$$
so that $p_n(0) = B_n(\lambda)$ and $q_n(0) = C_n(\lambda)$.
By Lemma \ref{coefficients estimate}, there is a huge integer $N$ such that
	$$|B_N(\lambda)|, |C_{N}(\lambda)|<(1+\eps/4)^{2^{N-1}}e^{\gamma(\lambda)2^{N-1}}$$
and
	$$\frac{\log (|\lambda|+3)}{2^{N-1}} \leq \eps/2.$$
Set
   $$R:=(1+\eps/4)^{2^{N-1}}e^{\gamma(\lambda)2^{N-1}}(|\lambda|+3).$$
Since $|B_{N}(\lambda)|, |C_{N}(\lambda)|< R/(|\lambda|+3)$, we can choose a very small $\delta>0$ such that
   $$ |p_{N}(s)|, |q_{N}(s)| < R/(|\lambda|+3),$$
for any $s$ with $|s|\leq \delta$.

Recall that $F_{n+1}^+(1,s) = F_{1/s}(F_n^+(1,s))$ for all $n$.  Thus,
$$|p_{N+1}(s)|=|\lambda  p_{N}(s) q_{N}(s)|< R^2/(|\lambda|+3)$$
and
$$|q_{N+1}(s)|=|s^2 p_N(s)^2+q_{N}(s)^2+p_{N}(s)q_{N}(s)|<R^2/(|\lambda|+3).$$
Inductively, we find
$$|p_{N+i}(s)|<R^{2^i}/(|\lambda|+3)$$
$$|q_{N+i}(s)|<R^{2^i}/(|\lambda|+3)$$
for any $s$ with $|s|\leq \delta$ and any $i\geq 0$. Consequently, as the integer $N$ satisfies $\log (|\lambda|+3)/2^{N-1}\leq \eps/2$ and  $|s|\leq \delta$,
\begin{eqnarray*}
\frac{\log\|F_{N+i,\lambda}(1,s)\|}{2^{N+i-1}}&<&  \frac{\log R^{2^i}/(|\lambda|+3)}{2^{N+i-1}}\\
   &<& \frac{ \log [(1+\eps/4)^{2^{N-1+i}}e^{\gamma(\lambda)2^{N-1+i}}(|\lambda|+3)^{2^i}]}{2^{N+i-1}}\\
   &<&\gamma(\lambda) +\eps/4+\log (|\lambda|+3)/2^{N-1}\\
   &<&\gamma(\lambda) +\eps.
\end{eqnarray*}
\qed

\bigskip
The corresponding lower bound on the size of $2^{1-n} \log\|F_n^+(1,s)\|$ for small $s$ is more delicate, and we use Lemma \ref{H bound} together with the estimates of Lemma \ref{coefficients estimate}.

\begin{lemma} \label{lower bound by epsilon}
For any given $\eps>0$, there exists a $\delta>0$ and an integer $N>0$ so that
	$$\frac{1}{2^{n-1}} \log \| F_n^+(1,s) \| - \gamma(\lambda) > -\eps$$
for all $|s| < \delta$ and all $n\geq N$.
\end{lemma}

\proof
In contrast with the proof of Lemma \ref{upper bound by epsilon}, we define polynomials $p_n(s)$ and $q_n(s)$ by
   $$F_n^+(1,s) = (s(B_n(\lambda)+sp_n(s)), C_n(\lambda)+sq_n(s)).$$
By Lemma \ref{coefficients estimate}, there is a huge integer $N$, such that
	$$|A_N(\lambda)|, |D_{N}(\lambda)|<(1+\eps/8)^{2^{N-1}}e^{\gamma(\lambda)2^{N-1}},$$
   $$|B_{N+i}(\lambda)|<((1+\eps/8)e^{\gamma(\lambda)})^{2^{N-1+i}},$$
and
\begin{equation}\label{Cn esitmate}((1-\eps/8)e^{\gamma(\lambda)})^{2^{N-1+i}}< |C_{N+i}(\lambda)|<((1+\eps/8)e^{\gamma(\lambda)})^{2^{N-1+i}},
\end{equation}
for any $i\geq 0$. By increasing $N$ if necessary, we may also assume that
\begin{equation}\label{C term}
   \frac{\log\left(c\eps/8(12|\lambda|+12)^{2}\right)}{2^{N-1}}>-\eps/10,
\end{equation}
where the constant $c$ is defined in Lemma \ref{H bound}.
Set
   $$R:=((1+\eps/8)e^{\gamma(\lambda)})^{2^{N-1}}(12|\lambda|+12).$$
Since $|A_{N}(\lambda)|, |D_{N}(\lambda)|<R/(12|\lambda|+12)$, we can choose a very small $\delta>0$ such that
   $$ |p_{N}(s)|, |q_{N}(s)|<R/(12|\lambda|+12),$$
for any $s$ with $|s|\leq \delta$.  Recalling that $F_{n+1}^+(1,s) = F_{1/s}(F_n^+(1,s))$ for all $n$, the estimate (\ref{Cn esitmate}) implies
$$|p_{N+1}(s)|=|\lambda (C_{N}(\lambda) p_{N}(s)+(B_{N}(\lambda) +s p_{N}(s))q_{N}(s))|<\frac{R^2}{12|\lambda|+12}$$
and similarly, $|q_{N+1}(s)|<R^2/(12|\lambda|+12)$.  Inductively, for any $i\geq 0$ and $s$ with $|s|\leq \delta$,
\begin{equation}\label{qn estimate}
|p_{N+i}(s)|, |q_{N+i}(s)| < R^{2^i}/(12|\lambda|+12).
\end{equation}

Choose an integer $N'>N$, such that
  \begin{equation}\label{delta 3}
  \delta':=\frac{\eps(1-\eps/8)^{2^{N'-1}}}{8(1+\eps/8)^{2^{N'-1}}(12|\lambda|+12)^{2^{N'-N}}}<\delta.
  \end{equation}
For any $j\geq N'$ and  $s$ with
   $$|s|\leq \frac{\eps(1-\eps/8)^{2^{j-1}}}{8(1+\eps/8)^{2^{j-1}}(12|\lambda|+12)^{2^{j-N}}}\leq\delta',$$
by (\ref{Cn esitmate}), we have
  \begin{eqnarray*}
  \frac{\log\|F_j(1,s)\|}{2^{j-1}}&\geq&\frac{\log |C_{j}(\lambda)+sq_j(s)|}{2^{j-1}}\\
                                  &\geq& \frac{\log( |C_{j}(\lambda)|-|sq_j(s)|)}{2^{j-1}}\\
                                  &\geq&\frac{\log\left( |(1-\eps/8)^{2^{j-1}}e^{\gamma(\lambda)2^{j-1}}|-\frac{|s|R^{2^{j-N}}}{12|\lambda|+12}\right)}{{2^{j-1}}}\textup{, by (\ref{qn estimate})}\\
                                  &\geq& \gamma(\lambda)+2\log (1-\eps/8)\\
                                  &>&\gamma(\lambda) -\eps \textup{, as $\eps$ is small.}
\end{eqnarray*}
For any $n\geq N'$ and $s$ with  $|s|<\delta'$,  if $|s|\leq \frac{\eps(1-\eps/8)^{2^{n-1}}}{8(1+\eps/8)^{2^{n-1}}(12|\lambda|+12)^{2^{n-N}}}$, then the above inequality guarantees
   $$\frac{\log\|F_n^+(1,s)\|}{2^{n-1}}> \gamma(\lambda) -\eps.$$
Otherwise, by (\ref{delta 3}), there is a $j$ with $N'\leq j<n$, such that
   $$\frac{\eps(1-\eps/8)^{2^{j}}}{8(1+\eps/8)^{2^{j}}(12|\lambda|+12)^{2^{j+1-N}}}\leq |s|\leq \frac{\eps(1-\eps/8)^{2^{j-1}}}{8(1+\eps/8)^{2^{j-1}}(12|\lambda|+12)^{2^{j-N}}}.$$

From Lemma \ref{H bound}, we have
\begin{equation} \label{apply H bound}
   \frac{1}{2^{(n+i)-1}}\log\|F_{n+i}^+(1,s)\|- \frac{1}{2^{n-1}}\log\|F_n^+(1,s)\|\geq \frac{1}{2^{n-1}}\log (c|s|),
\end{equation}
for all $s$ with $|s|\leq 1$ and all $n, i\geq 1$.  Indeed, note that $F_{n+1}^+(1,s) = F_{1/s}(F_n^+(1,s))$ for all $n$.  We see that
\begin{eqnarray*}
\frac{ \|F_{n+i}^+ (1,s) \|}{\|F_n^+(1,s)\|^{2^i} }
&=& \frac{ \|F_{n+i}^+ (1,s) \|}{\|F_{n+i-1}^+(1,s)\|^2 } \left(\frac{ \|F_{n+i-1}^+ (1,s) \|}{\|F_{n+i-2}^+(1,s)\|^2 }\right)^2 \cdots \left(\frac{ \|F_{n+1}^+ (1,s) \|}{\|F_n^+(1,s)\|^2 }\right)^{2^{i-1}}  \\
&\geq&  (c |s|)^{2^i-1}.
\end{eqnarray*}
Therefore,
\begin{eqnarray*}
\frac{\log\|F_n^+(1,s)\|}{2^{n-1}}&\geq& \frac{\log\|F_j^+(1,s)\|}{2^{j-1}}+2^{1-j}\log (c|s|)\\
                                                   &\geq& \gamma(\lambda)+2\log (1-\eps/8)
                                                  +2^{1-j}\log \frac{c\eps(1-\eps/8)^{2^{j}}}{8(1+\eps/8)^{2^{j}}(12|\lambda|+12)^{2^{j+1-N}} }\\
                                                   &\geq& \gamma(\lambda)+4\log (1-\eps/8)-2\log (1+\eps/8)-\eps/10, \textup{ by (\ref{C term})}\\
                                                   &>&\gamma(\lambda)-\eps, \textup{ as $\eps$ is small.}
\end{eqnarray*}

 \qed

\bigskip
\section{Non-archimedean potential functions}
\label{non-archimedean}

In this section, we prove a non-archimedean counterpart to Theorem \ref{convergence}.  If we assume $\lambda\in\Qbar$, many of the computations of the previous section hold (and simplify) for the nonarchimedean absolute values on the number field $k = \Q(\lambda)$.  As such, we may conclude that the bifurcation measures $\mu^\pm_\lambda$ are the archimedean components of a pair of {\em quasi-adelic measures}, equipped with continuous potential functions.  We use the term ``quasi-adelic" because the measures might be nontrivial at infinitely many places, though the associated height functions (defined as a sum of all local potentials, over all places of $k$) converge.

\subsection{Defining the potential functions at each place}
Let $k$ be a number field and let $\kbar$  denote a fixed algebraic closure of $k$.  (In this article, we always take $k=\Q(\lambda)$ for $\lambda\in \Qbar$.)  Any number field $k$ is equipped with a set $\cM_k$ of pairwise inequivalent nontrivial absolute values, together with a positive integer $N_v$ for each $v \in \cM_k$, such that
\begin{itemize}
\item for each $\alpha \in k^*$, we have $|\alpha|_v = 1$ for all but finitely many $v \in \cM_k$; and
\item  every $\alpha \in k^*$ satisfies the {\em product formula}
\begin{equation} \label{product formula}
\prod_{v \in \cM_k} |\alpha|_v^{N_v} \ = \ 1 \ .
\end{equation}
\end{itemize}
For each $v \in \cM_k$, let $k_v$ be the completion of $k$ at $v$, let $\kvbar$ be an algebraic closure of $k_v$, and let $\CC_v$ denote the completion of $\kvbar$.  We work with the norm
	$$\|(z_1, z_2)\|_v = \max\{|z_1|_v, |z_2|_v\}$$
on $(\C_v)^2$.  We let $\PP^{1,an}_v$ denote the Berkovich projective line over $\CC_v$, which is a canonically defined path-connected compact Hausdorff space containing $\PP^1(\CC_v)$ as a dense subspace.  If $v$ is archimedean, then $\CC_v \cong \CC$ and $\PP^{1,an}_v = \PP^1(\CC)$.  See \cite{BRbook} for more information.

With the $v$-adic norms and $t\in \C_v$, we can define $H_{\lambda,v}^\pm$ exactly as in the archimedean case,
	$$H_{\lambda,v}^{\pm}(t) = \lim_{n\to\infty} \frac{1}{2^n} \log \| F_t^n(\pm 1, 1) \|_v.$$
The definition extends naturally to the Berkovich affine line $\A^{1,an}_v$; see \cite[Chapter 10]{BRbook}.  Recall the definition of $F_n$ from \S\ref{F_n}, now defined on $(\C_v)^2$.

\begin{theorem}  \label{non-archimedean convergence}
Fix $\lambda \in \Qbar$ with $\lambda$ nonzero and not a root of unity, or set $\lambda=1$.  For each place $v$ of $k = \Q(\lambda)$, the limits
	$$\lim_{n\to\infty} \frac{1}{2^{n-1}}  \log \| F_n^\pm (t_1, t_2) \|_v$$
converge locally uniformly on $(\C_v)^2 \setminus\{(0,0)\}$ to continuous functions $G^\pm_v$ satisfying
	$$G_v^\pm(t_1,t_2) = \begin{cases}
 				2H^\pm_{\lambda,v}(t_1/t_2)+\log |t_2|_v &\text{ if } t_2\neq0,\\
				\log|t_1|_v+ \gamma_v(\lambda) &\text{ if } t_2=0.
			\end{cases}$$
The function $G_v^\pm$ extends uniquely to define a continuous potential function for a probability measure $\mu_v^\pm$ on $\P^{1,an}_v$.
\end{theorem}

Theorem \ref{non-archimedean convergence} is nearly identical to Theorem \ref{convergence}, except there is no longer a condition on the finiteness of $\gamma(\lambda)$, and the convergence holds at all places $v$.  This finiteness is guaranteed by the following lemma.

\begin{lemma} \label{product formula convergence}
For every algebraic number $\lambda$ which is not a root of unity, or for $\lambda=1$, the sum
	$$\gamma_v(\lambda) = \frac{1}{2} \sum_{i=1}^\infty \frac{1}{2^i} \log|1 + \lambda + \cdots + \lambda^i| _v$$
converges for all places $v$ of $k = \Q(\lambda)$.
\end{lemma}

\proof The statement follows from the product formula for the number field $k$.  First assume that $\lambda\not=1$.  Note that $(1+ \lambda + \cdots + \lambda^i)(1-\lambda) = 1-\lambda^{i+1}$, so it suffices to prove the convergence of the sum
	$$\sum_{i=0}^{\infty} \frac{1}{2^i} \log|1 -\lambda^i|_v$$
at all places $v$.

Let $v$ be a non-archimedean place of $k$.  If $|\lambda|_v < 1$, then $|1-\lambda^i|_v = 1$ for all $i$, and $\gamma_v$ is 0.  If $|\lambda|_v > 1$, then $|1-\lambda^i|_v = |\lambda|_v^i$ for all $i$, and again the sum converges.  Similarly for archimedean places, as long as $|\lambda|_v \not=1$, it is easy to see that the sum defining $\gamma_v$ converges.

For any $\lambda\in\Qbar$, there are only finitely many places of $k=\Q(\lambda)$ for which $|\lambda|_v > 1$.   For all such $v$, we have $|1-\lambda^i|_v$ growing as $|\lambda|_v^i$ as $i\to \infty$.  For all other $v$, the absolute value $|1-\lambda^i|_v$ is uniformly bounded above (by 1 if non-archimedean, and by 2 if archimedean).  Let $\ell$ be the number of archimedean places for which $|\lambda|_v \leq 1$.

Suppose $v$ is a place such that $|\lambda|_v = 1$.  It remains to show that $|1-\lambda^i|_v$ cannot get too small as $i\to\infty$.  As $1-\lambda^i \not=0$ for all $i$, the product formula for $k$ states that
	$$\prod_{w\in \cM_k} |1-\lambda^i|_w^{N_w} = 1.$$
Therefore,
	$$|1 - \lambda^i|^{N_v}_v = \frac{1}{ \prod_{w\not=v \in \cM_k} |1-\lambda^i|_w^{N_w} } \geq  \frac{1}{2^\ell \, \prod_{w: |\lambda|_w > 1} |1-\lambda^i|^{N_w}_w } \geq c^i $$
for some constant $c>0$ and all $i$.  It follows that the expression for $\gamma_v$ converges at this place $v$.

Finally, assume $\lambda = 1$.  Then the sum becomes
	$$\gamma_v(1) = \sum_{j=2}^\infty \frac{1}{2^j} \log|j|_v.$$
The expression clearly converges at the unique archimedean place $v=\infty$.  Setting $v = p$ for any prime $p$, we have $|j|_p \geq 1/j$ for all $j\in\N$; therefore, the sum is easily seen to converge also in this case.
\qed

\begin{remark}
For a given $\lambda$, the value $\gamma_v(\lambda)$ may be nonzero at infinitely many places $v$.  For example, $\gamma_v(1)$ is nonzero at {\em all} places $v$ of $\Q$.  The conclusion of Lemma \ref{product formula convergence} also appears in \cite[Lemma 4]{Herman: Yoccoz}. 
\end{remark}

\subsection{Proof of Theorem \ref{non-archimedean convergence}.}
Fix $\lambda\in\Qbar$, with $\lambda$ nonzero and not a root of unity; or let $\lambda=1$.  If $v$ is an archimedean place of the number field $k = \Q(\lambda)$, then the Theorem follows immediately from Theorem \ref{convergence} and Lemma \ref{product formula convergence} (for the finiteness of $\gamma_v(\lambda)$).

Now suppose that $v$ is a non-archimedean place of $k$.  The proof of convergence in the archimedean case shows {\em mutatis mutandis} that the convergence to $G^\pm_v$ is locally uniform for all places $v$.  A line-by-line analysis of the proof of Theorem \ref{convergence} shows that the proof uses nothing more than the triangle inequality and elementary algebra.  As such, the estimates can only be improved when the usual triangle inequality is replaced by the ultrametric inequality in the case of a non-archimedean absolute value.

The extension of $G_v^\pm$ to Berkovich space and the construction of the measure $\mu_v^\pm$ as its Laplacian are carried out exactly as in \cite[\S10.1]{BRbook}.
\qed

\medskip
There is one lemma (Lemma \ref{H bound}) used in the proof of Theorem \ref{convergence} that we will need again in the next section, in the non-archimedean setting.  We state it explicitly here.

\begin{lemma}  \label{constant c}
For each $\lambda\in \Qbar\setminus\{0\}$, there is constant $c>0$ so that
	 $$ \frac {\| F_t (z_1,z_2) \|_v}{\|(z_1,z_2)\|_v^2} \geq c |t|_v^{-1} $$
for all $|t|_v\geq 1$, all $(z_1, z_2)\neq (0,0)$ in $(\C_v)^2$, and all places $v$ of $k = \Q(\lambda)$.
\end{lemma}

\proof
Set
	$$c = \min\{\min\{|\lambda|_v/2: v \in \cM_k\}, 1/4\}.$$
Note that $c > 0$ since there are only finitely many places $v$ for which $|\lambda|_v<1$.  Now fix $v\in\cM_k$.  If $v$ is archimedean, then the proof is identical to that of Lemma \ref{H bound}.  If $v$ is non-archimedean, the estimate can be simplified a bit.  Indeed, assume that $z_2=1$ and $|z_1|_v\leq 1$.  Then $\|(z_1,z_2)\|_v = 1$, and we estimate the norm of $sF_{1/s}(z_1,1) = (s\lambda z_1, sz_1^2 + s + z_1)$ with $0 < |s|_v \leq 1$.
For each such $s$, either $|s\lambda z_1|_v\geq c|s|_v^2$, or $|z_1|_v < |s|_v/2$ in which case,
   $$|sz_1^2+s+z_1|_v = |s|_v >  c |s|_v^2.  $$
In either case, $\|F_{1/s}(z_1,z_2)\| \geq c|s|$ and the lower bound is proved.  The case of $|z_2|_v < |z_1|_v=1$ follows by symmetry, and the conclusion of the lemma is obtained from the homogeneity of $F_t$.
\qed

\bigskip
\section{The homogeneous bifurcation sets}
\label{sets}

In this section, we study the escape-rate functions $G_v^\pm$ of Theorems \ref{convergence} and \ref{non-archimedean convergence} and compute the homogeneous capacity of the sets
	$$K^\pm_{\lambda,v} =  \{z\in (\C_v)^2:  G_v^\pm(z)\leq 0\}.$$
We also provide a bound on the diameter of the sets $K^\pm_{\lambda, v}$ that will be used in our proof of Theorem \ref{equidistribution at all places} (and Theorem \ref{equidistribution}).

\subsection{The homogeneous capacity}
We will consider compact sets $K\subset \C^2$ that are circled and pseudoconvex:  these are sets of the form
	$$K = \{(z,w)\in \mathbb{C}^2: G_K(z,w) \leq 0\}$$
for continuous, plurisubharmonic functions $G_K :  \C^2\setminus\{(0,0)\}  \to \R$ such that
	$$G_K(\alpha z, \alpha w) = G_K(z, w) + \log|\alpha|$$
for all $\alpha\in\C^*$; see \cite[\S3]{D:lyap}.  Such functions are (homogeneous) potential functions for probability measures on $\P^1$ \cite[Theorem 5.9]{Fornaess:Sibony}.

Set $G_K^+=\max\{G_K,0\}$.  The Levi measure of $K$ is defined by
	$$\mu_K=dd^cG_K^+\wedge dd^cG_K^+.$$
It is known that $\mu_K$ is a probability measure supported on $\partial K=\{G=0\}$. The homogeneous capacity of $K$ is defined by
$$\capacity(K)=\exp\Big(\iint\log|\zeta\wedge \xi| d\mu_K(\zeta)d\mu_K(\xi)\Big).$$
This capacity was introduced in \cite{D:lyap} and shown in \cite{Baker:Rumely:equidistribution} to satisfy $\capacity(K) = (d_\infty(K))^2$, where $d_\infty$ is the transfinite diameter in $\C^2$.

To compute the capacity, suppose that
	$$F_n : \C^2 \to \C^2$$
is a sequence of homogeneous polynomial maps such that the resultant $\Res(F_n) \not=0$ for all $n$ and that
	$$\lim_{n\to\infty} \frac{1}{\deg(F_n)} \log\|F_n\|$$
converges locally uniformly in $\C^2 \setminus \{(0,0)\}$ to the function $G_K$.  The capacity $\capacity(K)$ may be computed as
\begin{equation}  \label{archimedean cap}
 	\capacity(K)=\lim_{n\rightarrow\infty} |{\Res}(F_n)|^{-1/\deg(F_n)^2}.
\end{equation}
If, in addition, the maps $F_n$ are defined over a number field $k$, and if the convergence holds for all absolute values $v$ in $\cM_k$, with limiting function $G_v: (\C_v)^2\setminus \{(0,0)\} \to \R$, then the same computation works at all places $v$.  That is,
\begin{equation} \label{cap at all places}
 	\capacity(K_v)=\lim_{n\rightarrow\infty} |{\Res}(F_n)|_v^{-1/\deg(F_n)^2}.
\end{equation}
See \cite{DWY:Lattes} for a proof; similar statements appear in \cite{DR:transfinite}.

\subsection{Homogeneous capacity of the bifurcation sets.}
Recall the definition of $\gamma(\lambda)$ from Theorem \ref{convergence}.

\begin{theorem}\label{cap}
For all $\lambda\in\C\setminus\{0\}$ such that $|\gamma(\lambda)|<\infty$, the set
	$$K^\pm_\lambda = \{z\in \C^2:  G_\lambda^\pm(z)\leq 0\}$$
is compact, circled, and pseudoconvex; its homogeneous capacity is
$$\capacity(K^+_{\lambda})=\capacity(K^-_{\lambda})=\frac{1}{|\lambda|^2}\prod_{j=1}^{+\infty}|1+\lambda+\cdots+\lambda^{j}|^{-3\cdot 4^{-j-1}}.$$
For all $\lambda\in\Qbar\setminus\{0\}$, not a root of unity, or for $\lambda=1$, we have
$$\capacity(K^+_{\lambda,v})=\capacity(K^-_{\lambda,v})=\frac{1}{|\lambda|_v^2}\prod_{j=1}^{+\infty}|1+\lambda+\cdots+\lambda^{j}|_v^{-3\cdot 4^{-j-1}}$$
for all places $v$ of the number field $k = \Q(\lambda)$.
\end{theorem}

\proof
Fix $\lambda\in\C\setminus\{0\}$ so that $|\gamma(\lambda)| < \infty$.  We will provide the proof for $K^+_\lambda$; the result for $K^-_\lambda$ follows by symmetry.  Recall the definition of $F_n^+$ from (\ref{F_n plus}).  That $K^+_\lambda$ is compact, circled and pseudoconvex follows immediately from the definition and continuity of $G^+_\lambda$, stated in Theorem \ref{convergence}.

To compute the capacity of $K^+_\lambda$, we give a recursive relation between $\Res(F_n)$ and $\Res(F_{n+1})$.  Consider the transformation $A_\lambda(t_1,t_2)=(\lambda t_1 t_2, t_1^2+t_2^2)$. Then ${\rm Res}(A_\lambda \circ F_n)$  takes the form
 $$\left|
\begin{array} {cccccccc}
0 & a_1& \cdots & a_{d-1} & a_d &0 &\cdots & 0\\
 0 & 0 & a_1& \cdots & a_{d-1} & a_d & \cdots &0\\
  & & & \cdots & & &��& \\
   0 & 0& \cdots & 0 & a_1& \cdots&  a_{d-1} & a_d\\
  b_0 & b_1&\cdots & b_{d-1} & b_d &0 &\cdots & 0\\
 0 & b_0 & b_1& \cdots & b_{d-1} & b_d & \cdots &0\\
  & & & \cdots & & &��& \\
   0 & 0& \cdots & b_0 & b_1&\cdots&  b_{d-1} & b_d
\end{array}
\right| $$
and $\Res(F_{n+1})$  takes the form
 $$\left|
\begin{array} {cccccccc}
0 & a_1& \cdots & a_{d-1} & a_d &0 &\cdots & 0\\
 0 & 0 & a_1& \cdots & a_{d-1} & a_d & \cdots &0\\
  & & & \cdots & & &��& \\
   0 & 0& \cdots & 0 & a_1& \cdots&  a_{d-1} & a_d\\
  b_0+\frac{a_1}{\lambda}& b_1+\frac{a_2}{\lambda}&\cdots & b_{d-1}+\frac{a_d}{\lambda} & b_d &0 &\cdots & 0\\
 0 & b_0+\frac{a_1}{\lambda} & b_1+\frac{a_2}{\lambda}& \cdots & b_{d-1}+\frac{a_d}{\lambda} & b_d & \cdots &0\\
  & & & \cdots & & &��& \\
   0 & 0& \cdots & b_0+\frac{a_1}{\lambda} & b_1+\frac{a_2}{\lambda}&\cdots&  b_{d-1}+\frac{a_d}{\lambda}& b_d
\end{array}
\right|,$$
where $b_0=C_n^2(\lambda)$ and $b_0+a_1/\lambda=C_{n+1}(\lambda)$, the coefficients defined in (\ref{coefficients}).  Therefore the resultants $\Res(F_n)$ satisfy:
\begin{eqnarray*}
 \Res(F_{n+1})&=&\frac{C_{n+1}(\lambda)}{C_n^2(\lambda)}\Res(A_\lambda \circ F_n)\\
&=&\frac{C_{n+1}(\lambda)}{C_n^2(\lambda)}\Res(A_\lambda)^{\deg(F_n)}\Res(F_n)^{2\deg(A_\lambda)}  \\
&=&\frac{C_{n+1}(\lambda)}{C_n^2(\lambda)}\Res(A_\lambda)^{2^{n-1}}\Res(F_{n})^{4}
\end{eqnarray*}
where the second equality follows from the decomposition property of resultants (see \cite[Proposition 6.1]{D:lyap}).

We may compute from (\ref{C coeff}) that $C_1(\lambda)=1,\ C_2(\lambda)=1+\lambda$, while for $n\geq3$,
\begin{equation} \label{C formula}
C_n(\lambda)=(1+\lambda)^{2^{n-3}}(1+\lambda+\lambda^2)^{2^{n-4}}\cdots (1+\lambda+\cdots+\lambda^{n-2})^{2^0}(1+\lambda+\cdots+\lambda^{n-1}).
\end{equation}
From the definition of resultant, we have $\Res(A_\lambda) =\lambda^2$ and $\Res(F_1) =-\lambda$. Thus the recursive relation becomes
$$\Res(F_{n+1})=\frac{1+\lambda+\cdots+\lambda^n}{1+\lambda+\cdots+\lambda^{n-1}} \; \lambda^{2^n}\Res(F_n)^{4}.$$
By induction, $\Res(F_2)=\lambda^6(1+\lambda)$, and for $n\geq3$, we have
\begin{equation} \label{resultant formula}
\Res(F_n)=\lambda^{2\cdot 4^{n-1}-2^{n-1}}(1+\lambda+\cdots+\lambda^{n-1})\prod_{j=1}^{n-2}(1+\lambda+\cdots+\lambda^{j})^{3\cdot 4^{n-2-j}}.
\end{equation}
From equation (\ref{archimedean cap}), we conclude that
\begin{eqnarray*}
\capacity(K^+_{\lambda})&=&\lim_{n\rightarrow\infty} |\Res(F_{n})|^{-1/\deg(F_{n})^2}\\
&=&\lim_{n\rightarrow\infty} \left(\frac{|\lambda|^{-2+2^{1-n}}}{|1+\lambda+\cdots+\lambda^{n-1}|^{4^{1-n}}}\prod_{j=1}^{n-2}|1+\lambda+\cdots+\lambda^{j}|^{-3\cdot 4^{-j-1}}\right)\\
&=&\frac{1}{|\lambda|^2}\prod_{j=1}^{+\infty}|1+\lambda+\cdots+\lambda^{j}|^{-3\cdot 4^{-j-1}}.
\end{eqnarray*}
The proof for $\lambda\in\Qbar$ and non-archimedean absolute values is identical, using (\ref{cap at all places}).
\qed

\subsection{Bounds for the homogeneous sets.}  \label{subsection bounds}
Fix $\lambda\in\Qbar$, not a root of unity (except possibly 1) and nonzero.  In our proof of Theorem \ref{equidistribution}, we will need control over the diameter of $K^\pm_{\lambda,v}$ at most places $v$ of the number field $k = \Q(\lambda)$.  We define subsets of the set of places $\cM_k$:
\begin{itemize}
\item for each $n\geq 1$, let $\cM_{k, n}$ be the set of all non-archimedean places $v$ for which $|\lambda|_v=1$, $|1 + \lambda + \cdots + \lambda^i|_v = 1$ for all $i< n$, and  $|1 + \lambda + \cdots + \lambda^n|_v<1$; and
\item  let $\cM_{k, 0}$ be the set of all non-archimedean places $v$ for which $|\lambda|_v = 1$ and $|1 + \lambda + \cdots + \lambda^i|_v=1$ for all $i$.
\end{itemize}
Note that the set $\cM_k \setminus  \bigcup_{n\geq 0} \cM_{k,n}$ is finite.  (Indeed, for non-archimedean $v$, there is an integer $i>0$ with $|1+\lambda+\cdots +\lambda^i|_{{v}}>1$ if and only if $|\lambda|_v>1$.  There are only finitely many such non-archimedean places. And there are only finitely many archimedean places.)  Also, the set $\cM_{k,0}$ might be empty, as will be the case for $\lambda=1$.

\begin{lemma} \label{trivial K}
For all $v$ in $\cM_{k,0}$, the sets $K_v^+$ and $K_v^-$ are trivial; that is,
	$$G_v^\pm(t_1,t_2) = \log\|(t_1, t_2)\|_v$$
and $K_v^\pm = \bar{D}^2(0,1)$.
\end{lemma}

\proof
For each $n\geq 1$, the coefficients of $F_n$ will have absolute value $\leq 1$.  From the formula for the resultant of $F_n$ in (\ref{resultant formula}), we see that $|\Res(F_n)|_v = 1$ for all $n$.  Applying \cite[Lemma 10.1]{BRbook} to $F_n$, setting $B_1=B_2=1$, we conclude that
	$$\|F_n(t_1, t_2)\| = \|(t_1, t_2)\|^{\deg F_n}$$
for all $n\geq 1$.  The conclusion follows immediately.
\qed

\begin{prop}\label{the bounds for K_v^+}
Fix $\lambda\in\Qbar\setminus\{0\}$ not a root of unity, or set $\lambda=1$.
There exists a constant $c = c(\lambda) > 0$ so that
$$\bar{D}^2(0,1) \subset K_v^+\subset \bar{D}^2(0, e^{-2\gamma_v (\lambda) -\frac {\log c}{2^{n-1}}}) $$
for all $v\in \cM_{k,n}$ and all $n \geq 1$.
\end{prop}

\proof
Fix $n\geq 1$ and $v\in\cM_{k,n}$.  For each $m\geq 1$, the coefficients of $F_m$ lie in the valuation ring of $k$ (i.e.~ have absolute value $\leq 1$).  It follows immediately that $\|F_m(t,s)\|_v \leq 1$ for all $m$ and all $\|(t,s)\|_v \leq 1$, and therefore, $G_v^+(t,s) \leq 0$ on $\bar{D}^2(0,1)$.  Consequently, $\bar{D}^2(0,1) \subset K_v^+$.

Suppose $|s|_v = 1$ and $|t|_v \leq 1$.  The resultant of the polynomial map $F_{\lambda,t}$ is $\lambda^2$, so it has $v$-adic absolute value $= 1$. The non-archimedean estimates of \cite[Lemma 10.1]{BRbook} imply that every iterate of $(z,w) = (1,1)$ under $F_{\lambda, t/s}$ will have norm 1.  Therefore, the $v$-adic norm of
 $$F_m(t, s) = s^{2^{n-1}} F^m_{\lambda, t/s}(1,1)$$
is also equal to $1$ for all $m$.  Consequently,  $G^+_v(t,s) = 0$ whenever $|t|_v \leq 1$ and $|s|_v = 1$.   


The coefficients $C_i$ of $F_i$, defined in (\ref{C coeff}) with formulas in (\ref{C formula}), satisfy $|C_i|_v=1$ for $i \leq n$ and $|C_i|_v<1$ for $i > n$.  In fact, from the explicit expressions for $C_i$, we see that $|C_i|_v$ form a non-increasing sequence with
	$$\lim_{i\to\infty} |C_i|_v = 0$$
and the sequence of expressions $ \frac{\log |C_i|_v}{2^{i-1}}$ also form a non-increasing sequence with
	$$\lim_{i\to\infty} \frac{\log |C_i|_v}{2^{i-1}} = \gamma_v(\lambda).$$
Fix $s\in\C_v$ with $|s|_v<1$,  and choose $j\geq n$ so that
    $$|C_{j+1}|_v\leq |s|_v <|C_j|_v. $$
Then by Lemma \ref{constant c} (applied exactly as in (\ref{apply H bound}) in the proof of Theorem \ref{convergence}), we have
    $$\frac{\log \|F_m^+(1,s)\|_v}{2^{m-1}}\geq \frac{\log \|F^+_j(1,s)\|_v}{2^{j-1}}+\frac{\log c|s|_v}{2^{j-1}}$$
for all $m \geq j$.
Since  $|C_{j+1}|_v\leq|s|_v <|C_j|_v$ and all the coefficients of $F_j(1,s)$ are bounded by 1, the constant term $C_j$ dominates the norm of $F_j(1,s)$ (that is, $\|F_j(1,s)\|_v=|C_j|_v$); hence
    $$\frac{\log \|F_m^+(1,s)\|_v}{2^{m-1}}\geq \frac{\log |C_j|_v}{2^{j-1}}+\frac{\log (c |C_{j+1}|_v)}{2^{j-1}}.$$
Letting $m\to \infty$, and since $\frac{\log |C_j|_v}{2^{j-1}} \geq \gamma_v(\lambda)$ for all $j$,  we have
    $$G_v^+(1,s)\geq 2\gamma_v (\lambda) +\frac {\log c}{2^{j-1}}.$$
It follows that
	$$G_v^+(1,s)\geq  2\gamma_v (\lambda) +\frac {\log c}{2^{n-1}} $$
for all $s$ with $|s|_v<1$, and we conclude that
	$$K_v^+\subset \bar{D}^2(0, e^{-2\gamma_v (\lambda) -\frac {\log c}{2^{n-1}}}).  $$
 \qed

\bigskip
\section{The equidistribution theorem}
\label{equidistribution section}

Throughout this section, we fix $\lambda\not= 0$ in $\Qbar$, and fix a number field $k$ containing $\lambda$.  As in \cite{Call:Silverman}, we define canonical height functions $\hat{h}^+$ and $\hat{h}^-$ on parameters $t\in \Per_1(\lambda)^{cm} (\Qbar)$, by
	$$\hat{h}^\pm(t) := \hat{h}_{f_t}(\pm1) = \lim_{n\to\infty} \frac{1}{2^n} h(f_t^n(\pm1)),$$
where $h$ is the logarithmic Weil height on $\P^1(\Qbar)$ and $\hat{h}_{f_t}$ is the canonical height of the morphism $f_t$.

\begin{theorem}  \label{equidistribution at all places}
Assume that $\lambda\in\Qbar\setminus\{0\}$ is not a root of unity, or set $\lambda=1$.  Let $\{S_n\}$ be any non-repeating sequence of $\Gal(\kbar/k)$-invariant finite sets in $\Per_1(\lambda)^{cm}$ for which
	$$\hat{h}^+(S_n) \to 0$$
as $n\to\infty$.  Then the sets $S_n$ are equidistributed with respect to the measure $\mu^+_\lambda$.  In fact, for each place $v$ of $k$, the discrete measures
	$$\mu_n = \frac{1}{|S_n|} \sum_{t\in S_n} \delta_t$$
converge weakly to the measure $\mu_v^+$ on the Berkovich projective line $\P^{1,an}_v$.   Similarly for $\hat{h}^-$ and the measures $\{\mu^-_v\}$.
\end{theorem}

\noindent
The main idea of the proof is to show that $\hat{h}^+$ and $\hat{h}^-$ are canonically associated to the ``quasi-adelic" measures $\{\mu^{\pm}_v\}$.  We use Theorems \ref{convergence} and \ref{non-archimedean convergence}.  Then we may apply the arithmetic equidistribution theorem (as appearing in \cite{Ye:quasi}, modified from the original treatements in \cite{BRbook, FRL:equidistribution}) to obtain the theorem.

\subsection{Quasi-adelic measures and equidistribution.}
For each $v \in \cM_k$ there is a distribution-valued Laplacian operator $\Delta$ on $\PP^1_{\Berk,v}$.  For example, the function $\log^+|z|_v$ on $\PP^1(\CC_v)$ extends naturally to a continuous real valued function $\PP^1_{\Berk,v} \backslash \{ \infty \} \to \RR$ and
	$$\Delta \log^+|z|_v = \lambda_v - \delta_{\infty},$$
where $\lambda_v$ is the uniform probability measure on the complex unit circle $\{ |z| = 1 \}$ when $v$ is archimedean and $\lambda_v$ is a point mass at the Gauss point of $\PP^1_{\Berk,v}$ when $v$ is non-archimedean.  (The sign of the Laplacian $\Delta$ is reversed from that of \cite{BRbook} or the presentation in \cite{BD:polyPCF}, to match the sign convention from complex analysis.)

A probability measure $\mu_v$ on $\PP^1_{\Berk,v}$ is said to have a {\em continuous potential} if $\mu_v - \lambda_v = \Delta g$ with $g : \PP^1_{\Berk,v} \to \RR$ continuous.  If $\mu_v$ has a continuous potential then there is a corresponding {\em  Arakelov-Green function} $g_{\mu_v} : \PP^1_{\Berk,v} \times \PP^1_{\Berk,v} \to \RR \cup \{ +\infty \}$ which is characterized by the differential equation $\Delta_x g_{\mu_v}(x,y) = \mu - \delta_y$ and the normalization
\begin{equation} \label{normalization}
	\iint g_{\mu_v}(x,y) d\mu(x) d\mu(y) = 0.
\end{equation}
Working with homogeneous coordinates, $g_{\mu_v}$ may be computed in terms of a continuous potential function for $\mu_v$,
	$$G_v:  (\C_v)^2\setminus \{(0,0)\} \to \R$$
satisfying $G_v(z_1, z_2) = g(z_1/z_2) + \log^+|z_1/z_2|_v + \log|z_2|_v$ for some continuous potential $g$ as described above.  For $x, y\in \P^1(\C_v)$, the Arakelov-Green function for $\mu_v$ is given by
\begin{equation}  \label{explicit g}
g_{\mu_v}(x,y) = -\log|\tilde{x} \wedge \tilde{y}|_v + G_v(\tilde{x}) + G_v(\tilde{y}) + \log \capacity(K_v),
\end{equation}
for any choice of lifts $\tilde{x}$ of $x$ and $\tilde{y}$ of $y$ to $(\C_v)^2$.  Here
	$$K_v = \{(a,b)\in (\C_v)^2: G_v(a,b) \leq 0\}.$$
The homogeneous capacity $\capacity(K_v)$ is exactly what is needed to normalize $g_{\mu_v}$ according to (\ref{normalization}).  See \cite[\S10.2]{BRbook} for details, in the setting where $K_v$ is the filled Julia set of a homogeneous polynomial lift of a rational function defined over $k$.

A {\em quasi-adelic measure} on $\PP^1$ (with respect to the field $k$) is a collection ${\mathbb \mu} = \{ \mu_v \}_{v \in M_k}$ of probability measures on $\P^{1,an}_v$ with continuous potentials for which the product
	$$\prod_{v\in M_k} ( r(K_v) / \capacity(K_v)^{1/2} )^{N_v}$$
converges strongly to a positive real number, where
	$$r(K_v) = \inf\{r>0: K_v \subset \bar{D}^2_v(0, r)\}$$
is the outer radius of $K_v$.  (Strong convergence of a product is, by definition, absolute convergence of the sum of logarithms of the entries.)

If $\rho,\rho'$ are measures on $\PP^{1,an}_v$, we define the
{\em $\mu_v$-energy} of $\rho$ and $\rho'$ by
$$( \rho, \rho' )_{\mu_v} := \frac{1}{2} \iint_{\PP^1_{\Berk,v} \times \PP^1_{\Berk,v} \backslash {\rm Diag}} g_{\mu_v}(x,y) d\rho(x) d\rho'(y).$$
Let $S$ be a finite, $\Gal(\kbar/k)$-invariant subset of $\PP^1(\kbar)$ with $|S|>1$.  For each $v \in \cM_k$, we denote by $[S]_v$ the discrete probability measure on $\PP^{1,an}_v$ supported equally on the elements of $S$.
For a quasi-adelic measure $\mu =\{\mu_v\}$, the {\em $\mu$-canonical height} of $S$ is defined by
\begin{equation}  \label{height definition}
\hhat_{{\mu}}(S) :=  \frac{|S|}{|S|-1} \sum_{v \in \cM_k} N_v \cdot ([S]_v,[S]_v)_{\mu_v}.
\end{equation}
The constants $N_v$ are the same as those appearing in the product formula (\ref{product formula}).

\begin{remark}  The definition of $\hhat_{\mu}$ differs slightly from that given in \cite{BD:polyPCF} or \cite{FRL:equidistribution}, but agrees with the definition in \cite{DWY:Lattes}; the factor of $|S|/(|S|-1)$ is included to match the usual definition of canonical height.  See Proposition \ref{same heights} and \cite[Lemma 10.27]{BRbook}.   With this normalization, the function $\hhat_\mu$ will extend naturally to sets with $|S|=1$, to define a function on $\P^1(\kbar)$.
\end{remark}

The following equidistribution theorem is a modification of the ones appearing in \cite{BRbook, FRL:equidistribution}; the proof is given in \cite{Ye:quasi}.

\begin{theorem} \label{quasi-adelic equidistribution}
Let $\hhat_{{\mu}}$ be the canonical height associated to a quasi-adelic measure $\mu$.
Let $\{S_n\}_{n\geq 0}$ be any non-repeating sequence of $\Gal(\kbar/k)$-invariant finite subsets of $\PP^1(\kbar)$ for which $\hhat_{{\mathbb \mu}}(S_n) \to 0$ as $n \to \infty$.  Then $[S_n]_v$ converges weakly to $\mu_v$ on $\PP^1_{\Berk,v}$ as $n \to \infty$ for all $v \in \cM_k$.
\end{theorem}

\subsection{Bifurcation measures are quasi-adelic.}
Now we prove that the escape-rate functions $G_v^\pm$ from Theorems \ref{convergence} and \ref{non-archimedean convergence} are potential functions for a quasi-adelic measure.  For each $n\geq 1$, recall from \S\ref{subsection bounds} that $\cM_{k, n}$ denotes the set of all non-archimedean places in $\cM_k$ such that $|\lambda|_v=1$,  $|1 + \lambda + \cdots + \lambda^i|_v=1$ for all $i< n$, and  $|1+ \lambda + \cdots + \lambda^n|_v<1$.

\begin{lemma}\label{places bound}
For any $\lambda\in \Qbar\setminus\{0\}$ not a root of unity, or for $\lambda=1$, there exists a constant $C = C(\lambda, k)$ so that
     $$|\cM_{k, n}| \leq \, C \, n$$
for all $n\geq 1$, where the places $v\in \cM_{k, n}$ are counted with multiplicity $N_v$.
\end{lemma}

\proof
We begin with a basic observation from algebraic number theory.  Let $m = [k:\Q]$, the degree of the field extension.  Suppose that $v$ is a place of $k$ extending the $p$-adic absolute value on $\Q$ for a prime $p$.  Then
      $$|x|_v<1 \; \implies \;  |x|_v\leq \frac{1}{p^{1/m}}$$
for all $x\in k$.  Indeed, the absolute value will be bounded by $p^{-1/e}$ where $e$ is the index of ramification of the field $k$ at the prime $p$; and $e\leq m$.

The proof of the lemma follows from the product formula and the above control on the absolute values.  There are only finitely many places $v\in \cM_k$ for which $|1+\lambda+ \cdots + \lambda^i|_v>1$ for some $i\geq 1$. Then there is an $M>1$ such that for any $n\geq 1$,
    $$\prod_{v\in \cM_k, |1+\lambda+ \cdots + \lambda^n|_v>1} |1+\lambda+ \cdots + \lambda^n|_v^{N_v} \leq M^n.$$
For each $v\in \cM_{k,n}$, we have
    $$|1+\lambda+ \cdots + \lambda^n|_v^{N_v}\leq \frac{1}{2^{\frac{N_v}{m}}}.$$
By the product formula, we see that
\begin{eqnarray*}
\prod_{v\in\cM_{k,n}}  |1+\lambda+ \cdots + \lambda^n|_v^{N_v} &\geq& \prod_{v,  |1+\lambda+ \cdots + \lambda^n|_v<1}  |1+\lambda+ \cdots + \lambda^n|_v^{N_v} \\
	&=& \prod_{v,  |1+\lambda+ \cdots + \lambda^n|_v>1}  |1+\lambda+ \cdots + \lambda^n|_v^{-N_v} \\
	&\geq& \frac{1}{M^n}.
\end{eqnarray*}
 Therefore,
    $$\frac{M^n}{2^{\frac{|\cM_{k,n}|}{m}}}\geq 1.$$
and we conclude that
    $$|\cM_{k,n}|\leq n m \log_2 M, $$
for any $n$.  Set $C = m\log_2M$.
\qed

\begin{prop} \label{quasi-adelic}
For each $\lambda\in\Qbar\setminus\{0\}$ not a root of unity, or for $\lambda=1$, the bifurcation measures $\{\mu^+_v\}$ and $\{\mu^-_v\}$ are quasi-adelic.
\end{prop}

\proof
As continuity of the potentials has already been established (Theorem \ref{non-archimedean convergence}), we need only show that the product
	$$\prod_{v\in \cM_k} ( r(K_v) / \capacity(K_v)^{1/2} )^{N_v}$$
converges strongly to a positive real number.  That is, we need to show the absolute convergence of the sum
	$$\sum_{v\in \cM_k} N_v  \log | r(K_v) / \capacity(K_v)^{1/2} |.$$
Lemma \ref{global} implies that $\prod \capacity(K_v)^{N_v}$ converges strongly to 1.  It remains to show the strong convergence of $\prod_v r(K_v)^{N_v}$.  Recall the definitions of $\cM_{k,n}$ and $\cM_{k,0}$ from \S\ref{subsection bounds}, and recall that the set of places not in $\cM_{k,0}$ or $\cM_{k,n}$ for any $n$ is finite.  From Lemma \ref{trivial K}, we have $K_v = \bar{D}^2(0,1)$ for all $v\in\cM_{k,0}$ so that $r(K_v) = 1$.   For $v\in \cM_{k,n}$, Proposition \ref{the bounds for K_v^+} shows that
	$$1 \leq r(K_v) \leq e^{-2\gamma_v(\lambda) - (\log c)/2^{n-1}}.$$
Strong convergence of $\prod_v r(K_v)^{N_v}$ will then follow from convergence of
	$$\sum_{n=1}^\infty \sum_{v\in \cM_{k,n}}  N_v \left(-2 \gamma_v(\lambda) - \frac{\log c}{2^{n-1}} \right).$$
Lemma \ref{global} implies that the sum of the $N_v \gamma_v(\lambda)$ terms will converge.  Lemma \ref{places bound} shows that $|\cM_{k,n}| \leq C n$ when counted with multiplicities $N_v$, showing that the sum of the $(\log c)/2^{n-1}$ terms will also converge.
\qed

\subsection{Equivalence of two canonical heights.}  Now we show that the Call-Silverman heights $\hat{h}^\pm$, defined at the beginning of \S\ref{equidistribution section}, coincide with the $\mu^\pm$-canonical heights associated to the quasi-adelic measures $\{\mu^\pm_v\}$, defined in (\ref{height definition}).  We begin with a lemma.
In other settings (e.g.~those in \cite[Chapter 10]{BRbook}, \cite{BD:polyPCF} or \cite{DWY:Lattes}), the analogous conclusion of Lemma \ref{global} would be immediate from the product formula, since all but finitely many terms would be 0.  Recall that the definition of $\gamma_v(\lambda)$ appears in Lemma \ref{product formula convergence} and the formula for the capacity is given in Theorem \ref{cap}.

\begin{lemma}  \label{global}
Fix $\lambda\in\Qbar\setminus\{0\}$ not a root of unity, or set $\lambda=1$.  The sums
 $$\sum_{v\in\cM_k} N_v \gamma_v(\lambda) \qquad \mbox{ and } \qquad \sum_{v\in\cM_k} N_v \log\capacity(K_v^\pm)$$
 converge absolutely, and the sum is equal to 0.
\end{lemma}

\proof
There are only finitely many places $v$ for which there is an integer $i > 0$ with $|1+\lambda+\cdots +\lambda^i|_{{v}}>1$.  (For non-archimedean $v$, this condition is equivalent to $|\lambda|_v>1$; there are only finitely many such non-archimedean places.)  Therefore, there exists $M>1$ such that
    $$\prod_{v\in \cM_k, |1+\lambda+ \cdots + \lambda^n|_v>1} |1+\lambda+ \cdots + \lambda^n|_v^{N_v} \leq M^n$$
for all $n\geq 1$.  From the product formula, we see that
  $$\prod_{v\in \cM_k, |1+\lambda+ \cdots + \lambda^n|_v<1} |1+\lambda+ \cdots + \lambda^n|_v^{N_v} \geq \frac{1}{M^n},$$
though the number of such places grows with $n$.  Consequently,
	     $$\sum_{n=1}^{\infty} \frac{1}{r^n} \sum_{v\in \cM_k} N_v \left| \log|1+\lambda+\cdots +\lambda^n|_v \right|<\infty$$
for any constant $r>1$.    Therefore, we can interchange the order of summation and deduce that
	$$\sum_v N_v \sum_{n=1}^\infty \frac{1}{r^n}  \log|1+\lambda+\cdots +\lambda^n|_v =
	 \sum_{n=1}^\infty   \frac{1}{r^n}  \sum_v N_v \log|1+\lambda+\cdots +\lambda^n|_v = 0. $$
For $r=2$, we obtained the desired convergence for $\gamma_v(\lambda)$, and for $r=4$ the sum of the capacities.
\qed

\begin{prop} \label{same heights}
For each $\lambda\in\Qbar\setminus\{0\}$, not a root of unity, or for $\lambda=1$, the Call-Silverman canonical height $\hat{h}^+$ and the $\{\mu^+_v\}$-canonical height $\hat{h}_\mu$ are related by
	$$ \hat{h}_\mu(S) = \frac{2\, [k:\Q]}{|S|} \sum_{t\in S} \hat{h}^+(t)$$
for any $\Gal(\kbar/k)$-invariant, finite set $S$ with $|S|>1$.   Similarly for $\hat{h}^-$ and the $\{\mu^-_v\}$-canonical height.
\end{prop}

\proof
Fix a finite set $S\subset \kbar$ which is $\Gal(\kbar/k)$-invariant and has at least two elements. We begin by computing the $\{\mu^+_v\}$-canonical height of $S$, from the definition given in (\ref{height definition}).  For each $t\in S$, we choose a lift $\tilde{t} \in \kbar^2$ of $t$.

\begin{eqnarray*}
\hhat_{\mu}(S) &=& \frac{|S|}{|S|-1}\sum_{v\in \cM_k} N_v \cdot ([S]_v, [S]_v)_{\mu_v} \\
	&=&  \frac{|S|}{|S|-1}  \sum_{v\in \cM_k} \frac{N_v }{2|S|^2}\sum_{x,y\in S, x\not=y}  g_{\mu_v}(x,y) \\	
	&=&   \frac{1}{2|S|(|S|-1)} \sum_{x,y\in S, x\not=y}   \sum_{v\in \cM_k} N_v \, ( -\log|\tilde{x} \wedge \tilde{y}|_v + G^+_{v}(\tilde{x}) + G^+_{v}(\tilde{y}) + \log \capacity(K^+_{\lambda,v}) ) \\
	&=&  \frac{1}{2|S|(|S|-1)} \sum_{x,y\in S, x\not=y}  \; \sum_{v\in \cM_k} N_v \, (G^+_{v}(\tilde{x}) + G^+_{v}(\tilde{y}) )  \quad \mbox{ by (\ref{product formula}) and Lemma \ref{global}} \\
	&=& \frac{1}{2|S|(|S|-1)} \sum_{v\in \cM_k} N_v  \cdot 2\, (|S|-1)\sum_{x\in S} G^+_{v}(\tilde{x}) \\
	&=&  \frac{1}{|S|} \sum_{v\in \cM_k} N_v \cdot \sum_{x\in S} G^+_{v}(\tilde{x})\\
     &=&  \frac{1}{|S|}  \cdot \sum_{x\in S}\sum_{v\in \cM_k} N_v G^+_{v}(\tilde{x})
	\end{eqnarray*}
Note that the product formula (and homogeneity of $G_v^+$) implies that $\sum_v G_v^+(\tilde{x})$ depends on $x$ but is independent of the choice of $\tilde{x}$.  This formula for $\hhat_\mu$ also shows that it extends to define a function on points of $\P^1(k)$, as mentioned in the remark after equation (\ref{height definition}). 

Recall that the Weil height of $\alpha\in \kbar$ may be computed as
$$h(\alpha) = \frac{1}{[k(\alpha):\Q]} \sum_{v\in\cM_{k(\alpha)}} N_v \log \|(\alpha_1, \alpha_2)\|_v = \frac{1}{[k(\alpha):\Q]} \sum_{v\in\cM_{k(\alpha)}} N_v \log\max\{|\alpha_1|_v, |\alpha_2|_v\}$$
for any homogeneous presentation $(\alpha_1, \alpha_2)\in k(\alpha)^2$ of $\alpha$ (so $\alpha = \alpha_1/\alpha_2$).  Further, if $A$ is a $\Gal(\kbar/k)$-invariant and finite set, then
	$$h(A) := \frac{1}{|A|} \sum_{\alpha\in A} h(\alpha) = \frac{1}{[k:\Q]} \frac{1}{|A|} \sum_{\alpha\in A} \sum_{v\in\cM_{k}} N_v \log\max\{|\alpha_1|_v, |\alpha_2|_v\},$$
where $|\cdot|_v$ is a fixed choice of extension to $\kbar$ of the absolute value $|\cdot|_v$ on $k$.

For each $t\in S$, from the definition in (\ref{F_n plus}), the point $F^+_n(\tilde{t})\in \kbar^2$ is a homogeneous presentation of $f_t^n(+1)\in \kbar$.  The Galois invariance of $S$ implies that
\begin{equation} \label{limit}
\hat{h}^+(S) = \frac{1}{|S|} \sum_{t\in S} \hat{h}^+(t)
=\frac{1}{[k:\Q]} \frac{1}{|S|} \sum_{t\in S} \lim_{n\rightarrow\infty}\frac{1}{2^n}\sum_{v\in \mathcal{M}_k} N_v \log\|F_n(\tilde{t})\|_v.
\end{equation}
We need to show that we can interchange the limit and the infinite sum over $\cM_k$.  Then Theorems \ref{convergence} and \ref{non-archimedean convergence} will imply that
	$$\hat{h}^+(S) = \frac{1}{2 |S| [k:\Q]} \sum_{t\in S} \sum_{v\in \cM_k} N_v G_v^+(\tilde{t}) = \frac{1}{2 [k:\Q]}\, \hat{h}_\mu (S),$$
completing the proof of the theorem.

To see that we may interchange the limit and the sum in (\ref{limit}), we can use Lemma \ref{trivial K} and Proposition \ref{the bounds for K_v^+}.   Outside of a finite number of places, we have $\|\tilde{t}\|_v = 1$ for all $t\in S$.  For $v\in \cM_{k,0}$, Lemma \ref{trivial K} states that $\|F_n(\tilde{t})\|_v = 1$ for all $n$ when $\|\tilde{t}\|_v=1$, and so these terms do not contribute to the sum of (\ref{limit}).  For $v\in \cM_{k,m}$ with $m\geq 1$, Proposition \ref{the bounds for K_v^+} implies that
	$$e^{2^{n-1}(2\gamma_v(\lambda) + (\log c)/2^{m-1})} \leq \|F_n(\tilde{t})\|_v \leq 1$$
for all $n\geq 1$, when $\|\tilde{t}\|_v = 1$.  Since the sum of the $\gamma_v(\lambda)$ converges absolutely (Lemma \ref{global}) and there aren't too many elements in each $\cM_{k,m}$ (Lemma \ref{places bound}), we deduce that for any $\eps>0$, there is a finite set $\cM(\eps)$ of places so that
	$$\frac{1}{2^n} \sum_{v\in \cM_k \setminus \cM(\eps)} N_v \left| \log \|F_n(\tilde{t}) \|_v \right| < \eps$$
for every $n\geq 1$.  This is enough to allow the exchange of the limit with the sum in (\ref{limit}).

The proof for $\hat{h}^-$ and the $\{\mu^-_v\}$-canonical height is identical.
\qed

\subsection{Proofs of the equidistribution theorems.}  We are ready to complete the proofs of the equidistribution theorems.

\medskip\noindent{\em Proof of Theorem \ref{equidistribution at all places}.}
The proof is immediate from Theorem \ref{quasi-adelic equidistribution}, once we know that the canonical height functions $\hat{h}^+$ and $\hat{h}^-$ are associated to the quasi-adelic measures $\{\mu^\pm_v\}$.  That is the content of Propositions \ref{quasi-adelic} and \ref{same heights}.
\qed

\medskip\noindent{\em Proof of Theorem \ref{equidistribution}.}
This theorem is an immediate corollary of Theorem \ref{equidistribution at all places}, because the parameters where the critical point $\pm 1$ has finite orbit coincide with the points of canonical height 0, by \cite[Corollary 1.1.1]{Call:Silverman}.
\qed

\bigskip
\section{Proof of Theorem \ref{PCF maps}}
\label{proof section}

In this section, we complete the proof of Theorem \ref{PCF maps}.  As discussed in the introduction, one implication is well known; namely, the family of quadratic polynomials $\Per_1(0)$ contains infinitely many postcritically-finite maps.  Indeed, as explained in Example \ref{quadratic poly}, the bifurcation locus for the (non-fixed) critical point is nonempty.  The conclusion then follows from Lemma \ref{activity}.  

For dynamical reasons, there can be no postcritically-finite maps in $\Per_1(\lambda)$ for $0 < |\lambda|\leq 1$.  Indeed, if $f$ has a fixed point of multiplier $\lambda$, at least one critical point must have infinite forward orbit, as it is attracted to (or accumulates upon) the fixed point (or on the boundary of the Siegel disk in case the fixed point is of Siegel type).  See \cite[Corollary 14.5]{Milnor:dynamics}.  (Actually, for our proof of Theorem \ref{PCF maps}, we only need the easier fact that parabolic cycles always attract a critical point with infinite forward orbit \cite[\S10]{Milnor:dynamics}, since all other cases follow from the arguments below.)

Now assume that $\lambda\in\C$, $|\lambda| >1$, is chosen so that $\Per_1(\lambda)$ contains infinitely many postcritically-finite maps.  We shall derive a contradiction.  We first observe that $\lambda\in\Qbar$, as a consequence of Thurston Rigidity.

\begin{prop} \label{algebraic}
If  $f_{\lambda, t}(z) = \lambda z / (z^2 + t z + 1)$ is postcritically finite, then $\lambda, t\in \Qbar$.
\end{prop}

\proof
The critical points of a postcritically-finite map satisfy two equations
$$f_{\lambda, t}^n(+1) = f_{\lambda, t}^m(+1) \qquad \mbox{ and } \qquad f_{\lambda, t}^r(-1) = f_{\lambda, t}^s(-1)$$
for pairs of integers $n>m\geq0$ and $r>s\geq 0$.  Note that these two equations define polynomials in $(\lambda, t)$ with coefficients in $\Q$.  By Thurston Rigidity (see \cite[Theorem 2.2]{McMullen:families}), we know that the set of solutions must be finite (or empty), since there are no flexible Latt\`es maps in degree 2.  Consequently, a solution $(\lambda,t)$ must have coordinates in $\Qbar$.
\qed

\medskip
Let $k = \Q(\lambda)$.  Then $k$ is a number field, by Proposition \ref{algebraic}.  Construct the Call-Silverman canonical height functions $\hat{h}^+_\lambda$ and $\hat{h}^-_\lambda$ on $\kbar$ as in Section \ref{equidistribution section}.  Let $\{t_n\}_{n\in\N}$ denote a sequence of parameters in $\Per_1(\lambda)^{cm}$ for which both critical points have finite orbit.  Proposition \ref{algebraic} also shows that $t_n\in\kbar$ for all $n$.  Let $S_n$ denote the $\Gal(\kbar/k)$-orbit of $t_n$.  From \cite[Corollary 1.1.1]{Call:Silverman}, we see that
	$$\hat{h}^+_\lambda(S_n) = \hat{h}^-_\lambda(S_n) = 0$$
for all $n$.  By Theorem \ref{equidistribution at all places}, the sequence $\{S_n\}$ is equidistributed with respect to the bifurcation measures $\mu^+_v$ and $\mu^-_v$ at all places $v$ of $k$.  It follows that $\mu^+_v = \mu^-_v$ for all $v$, and therefore the quasi-adelic height functions (defined in (\ref{height definition})) must coincide.  From Proposition \ref{same heights}, we conclude that
	$$\hat{h}^+_\lambda = \hat{h}^-_\lambda$$
on $\Per_1(\lambda)^{cm}$.  Again appealing to \cite[Corollary 1.1.1]{Call:Silverman}, we find that the critical point $+1$ will have finite orbit for $f_t$ if and only if $-1$ has finite orbit for $f_t$.  This conclusion contradicts Proposition \ref{no synchrony}.  The proof is complete.

\begin{remark}
In our proof, we have used the full strength of Theorem \ref{equidistribution at all places}, with the equidistribution at all places of the number field $k = \Q(\lambda)$ to deduce that $\hat{h}^+ = \hat{h}^-$.  Alternatively, we could have used only Theorem \ref{equidistribution}, the equidistribution to the (complex) bifurcation measure $\mu^+_\lambda$, and deduced that $\mu^+_\lambda = \mu^-_\lambda$.  This conclusion would contradict Theorem \ref{distinct measures}.
\end{remark}

\bigskip \bigskip
\def\cprime{$'$}

\end{document}